\documentclass[twoside, final]{amsart} 

\usepackage{amsmath, amsfonts, amsthm, amssymb}
\usepackage{exscale}
\usepackage{cite}
\usepackage{epsfig}
\usepackage{amscd}
\usepackage{graphics}
\usepackage{color}
\usepackage{dsfont}

\usepackage[notref, notcite]{showkeys}

\renewcommand{\geq}{\geqslant}
\renewcommand{\leq}{\leqslant}

\newtheorem{theorem}{Theorem}

\newtheorem{proposition}{Proposition}[section]

\newtheorem{lemma}[proposition]{Lemma}
\newtheorem*{main-theorem}{Main Theorem}
\newtheorem*{theorem*}{Theorem}

\theoremstyle{definition}

\newtheorem{remark}[proposition]{Remark}

\newtheorem*{remark*}{Remark}

\numberwithin{equation}{section}

\def\phi{\varphi}

\def\reals{{\mathbb R}}

\def\Ci{{\mathcal C}^\infty}

\def\WF{\mathrm{WF}}

\def\supp{\mathrm{supp}\,}

\def\O{{\mathcal O}}

\def\phi{\varphi}

\def\be{\begin{eqnarray*}}
\def\ee{\end{eqnarray*}}
\def\ben{\begin{eqnarray}}
\def\een{\end{eqnarray}}

\def\L2R{L_{\text{Rest}}^2}

\def\11{\mathds{1}}

\def\tpsi{\tilde{\psi}}

\def\L2c{L^2_{\text{comp}}}

\def\tDelta{\widetilde{\Delta}}

\def\tP{\widetilde{P}}

\def\tR{\tilde{R}}

\def\tE{\widetilde{E}}
\def\p{\partial}

\def\GG{\mathcal{G}}

\def\tY{\tilde{Y}}
\def\tOmega{\tilde{\Omega}}

\begin{document}

\title[Non-concentration]{A non-concentration estimate for partially rectangular billiards}

\author[H. Christianson]{Hans Christianson}

\email{hans@math.unc.edu}
\address{Department of Mathematics, UNC-Chapel Hill \\ CB\#3250
  Phillips Hall \\ Chapel Hill, NC 27599}

\subjclass[2010]{}
\keywords{}

\begin{abstract}

We consider quasimodes on planar domains with a partially rectangular
boundary.  We prove that for any $\epsilon_0>0$, 
an $\O( \lambda^{-\epsilon_0})$ quasimode must have $L^2$ mass in the
``wings'' (in phase space) bounded below by $\lambda^{-2-\delta}$ for any $\delta>0$.
The proof uses the  author's recent work on 0-Gevrey smooth domains to
approximate quasimodes on $C^{1,1}$ domains.  There is an 
improvement for $C^{k,\alpha}$ and $C^\infty$ domains.

\end{abstract}

\maketitle

\section{Introduction}
\label{S:intro}



In this paper, we consider the famed Bunimovich stadium (and
similar partially rectangular billiards) and prove that weak
quasimodes must spread into the ``wings'' of the domain (at least in
phase space).  This type of
result is not new, however 
the lower bound on the quasimode mass in the
wings is a significant improvement over what is previously known, and
the additional phase space information appears to be new.  

We begin by describing the geometry.  
Let $\Omega \subset \reals^2$ be a planar domain with $C^{1,1}$ (or
$C^{k, \alpha}$ for $k + \alpha >2$), piecewise
$C^\infty$ boundary $\Gamma = \p \Omega$, and let $R = [-a,a]
\times [-\pi, \pi] \subset \reals^2$ be a rectangle with boundary
consisting of the two sets of parallel segments $\p R = \Gamma_1 \cup
\Gamma_2$, with $\Gamma_1 = [-a,a] \times \{ \pi \} \cup [-a,a] \times
\{ - \pi \}$.  Assume $R \subset \Omega$ and $\Gamma_1 \subset \p
\Omega$ but $\mathring{\Gamma}_2 \cap \p \Omega = \emptyset$.  We
assume that for $(x, y)$ in a neighbourhood of $R$, $\Gamma = \p
\Omega$ is symmetric about the line $y = 0$.  Let $Y(x) = \pi + r(x)$ be a graph parametrization of the boundary curve $\p \Omega$ for
$(x,y)$ near $[-a,a] \times \{ \pi \}$, that is the upper boundary
near the rectangular part.  In order to make what follows nontrivial,
let us assume that $r(x)$ is a $C^{1,1}$, piecewise $C^\infty$ function
with $r''(x) \neq 0$ for $\pm x \geq a$.  That is, the wings open or
close as you move away from the rectangular part.  For example, the
famed Bunimovich stadium is $C^{1,1}$, satisfying these assumptions  with $r(x) = (\pi^2 -
(x+a)^2)^{1/2} - \pi $ for $x$ to the left of the rectangular part
(see Figure \ref{fig:stadium}).



We consider quasimodes 
near the rectangular part, and show that the
mass in phase space in the wings is bounded below by $\lambda^{-2-\delta}$ for any
$\delta>0$.   

\begin{theorem}

\label{T:stad-non-conc}
Suppose $\Gamma$ is $C^{1,1}$ and for some $\epsilon_0>0$ $u$
satisfies the equation
\begin{align}
\label{E:qm-eqn}
\begin{cases}
-\Delta u = \lambda^2 u + E(\lambda ) \| u \|, \\
u |_\Gamma = 0,
\end{cases}
\end{align}
where $E(\lambda) = \O(\lambda^{-\epsilon_0} )$.
Then for any $\delta>0$, there exists $c_\delta>0$ such
that at least one of the following three inequalities is true: 
\begin{equation}
\label{E:lower-1}
\| u \|_{L^2(\Omega \setminus R)}  \geq  c_\delta \lambda^{-2-\delta}
\| u \|_{L^2(\Omega) },
\end{equation}
\begin{equation}
\label{E:lower-2}
\| \lambda^{-1} D_x u \|_{L^2( \Omega \setminus R)} \geq c_\delta \lambda^{-2-\delta}
\| u \|_{L^2(\Omega) },
\end{equation}
or 
\begin{equation}
\label{E:lower-3}
\| (\lambda^{-1} D_x)^2 u \|_{L^2( \Omega \setminus R)}  \geq c_\delta \lambda^{-2-\delta}
\| u \|_{L^2(\Omega) }.
\end{equation}

\end{theorem}

\begin{remark}
The estimate \eqref{E:lower-1} gives a lower bound on the $L^2$ mass
in the wings, while the estimates (\ref{E:lower-2}-\ref{E:lower-3})
give lower bounds on the mass in phase space, since the quasimode
equation tells us the function $u$ is already semiclassically
localized to the cosphere bundle.  Moreover, if one of
(\ref{E:lower-2}-\ref{E:lower-3}) is true, it is expected there is
some lower bound on $u$ in the wings as well, since having a large
$x$-derivative suggests there is lateral propagation.  We hope to
explore this further in later works.

The proof has a control theory flair to it; it goes by a contradiction
argument considering the mass in a $\lambda$-dependent strip just
outside the rectangular part.  
\end{remark}

\begin{remark}
We remark that the real difficulty in improving such estimates 
is the lack of regularity at the boundary of
$R$.  That is, if $\Gamma$ is smoother, we can 
improve the above estimates.  We prove a general result for $C^{k, \alpha}$
domains in Theorem \ref{C:stad-non-conc-2} below.  For a more extreme example, see
\cite[Theorem 3]{Chr-inf-deg} for a case with 0-Gevrey regularity.
\end{remark}

\begin{remark}
Theorem \ref{T:stad-non-conc} improves on the current state of the art
for quasimodes 
in 
\cite{BHW-spread} (see also \cite{BuZw-ball}) by improving $\lambda^{-4}$ to $\lambda^{-2-\delta}$.
We remark that for ``honest'' eigenfunctions (as opposed to
quasimodes), in \cite{BHW-spread} a $\lambda^{-2}$ lower bound is
proved, and in 
\cite{HiMa-p-rect}, this is improved to $\lambda^{-5/3 -}$ under an
additional 
spectral non-resonance assumption.
\end{remark}

\begin{figure}
\hfill
\centerline{\input{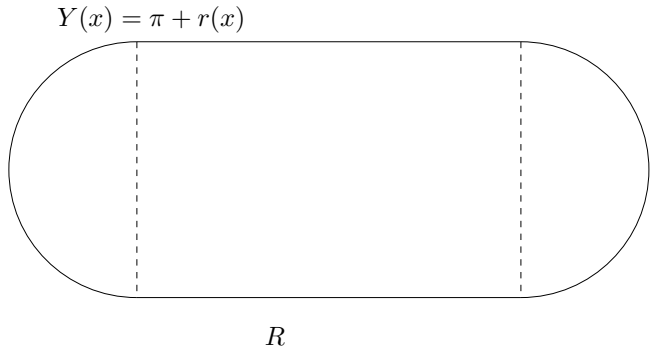}}
\caption{\label{fig:stadium} The  Bunimovich stadium.  The
  rectangular part is in the middle.  The boundary is $C^\infty$
  except at the four corners of the rectangle, where it is $C^{1,1}$.  The
main result of this work is that quasimodes must spread into the
semicircular ``wings'' outside the rectangle.}
\hfill
\end{figure}

In the following Theorem, we improve the lower bound in the case of
a $C^{k, \alpha}$ boundary for $0 \leq \alpha \leq 1$.  The improvement
is that the exponent in the lower bound will be smaller than $2$
provided $k +\alpha >2$.

\begin{theorem}

\label{C:stad-non-conc-2}
In addition to the assumptions of Theorem \ref{T:stad-non-conc},
assume the boundary $\Gamma$ is $C^{k, \alpha}$ for some $0 \leq
\alpha \leq 1$, $k + \alpha>2$.  
For  $\epsilon_0>0$, 
let $u$ satisfy
\[
\begin{cases}
-\Delta u = \lambda^2 u + E(\lambda ) \| u \|, \\
u |_\Gamma = 0,
\end{cases}
\]
where $E(\lambda) = \O(\lambda^{-\epsilon_0} )$.

Set
\[
s_\delta = 1+ \max \left\{ \frac{1}{k + \alpha}, \frac{1 +
    \delta}{2(k + \alpha) -3} \right\} + \delta.
\]
Then for any $\delta>0$, 
there exists $c = c_\delta >0$ such that at least one of the following
three 
inequalities holds true:
\begin{equation}
\label{E:lower-1a}
\| u \|_{L^2(\Omega \setminus R)}  \geq  c_\delta \lambda^{-s_\delta}
\| u \|_{L^2(\Omega) },
\end{equation}
\begin{equation}
\label{E:lower-2a}
\| \lambda^{-1} D_x u \|_{L^2(\Omega \setminus R)} \geq c_\delta \lambda^{-s_\delta}
\| u \|_{L^2(\Omega) },
\end{equation}
or
\begin{equation}
\label{E:lower-3a}
\| (\lambda^{-1} D_x)^2 u \|_{L^2(\Omega \setminus R)} \geq c_\delta \lambda^{-s_\delta}
\| u \|_{L^2(\Omega) }.
\end{equation}

In particular, if $\Gamma$ is $C^\infty$, then at least one of
(\ref{E:lower-1a}-\ref{E:lower-3a}) holds with $s_\delta = 1+\delta$ for
any $\delta>0$.


\end{theorem}

\begin{remark}
We pause to remark that the lower bound of $\lambda^{-1-\delta}$ in
the $C^\infty$ case agrees with the lower bound on $L^2$ mass in
rotationally invariant neighbourhoods on 0-Gevrey surfaces of
revolution proved by the author in \cite{Chr-surf-rev}.

\end{remark}

\subsection*{Acknowledgements}
The author would like to thank Michael Taylor for pointing out a
mistake in notation in an earlier version, as well as suggesting
writing out the improvement in the $C^{k,\alpha}$ case.  He would also
like to thank Luc Hillairet for pointing out a mistake in an earlier
version of this paper - indeed this is a much more delicate problem
than he initially thought! 
The author is 
supported in part by NSF grant DMS-0900524.

\section{History of the problem}
The study of eigenfunctions in partially rectangular domains,
especially the Bunimovich stadium, is interesting and important for
many reasons.  The main reason why one might be interested in the
properties of eigenfunctions on the Bunimovich stadium is the
elegant simplicity of the domain.  It has enough symmetries that one
might expect to use to simplify the problem, yet the eigenfunctions have no
known closed form, nor do quasimodes take any standard form.  

The broken geodesic flow (or billiard map) on the Bunimovich stadium
is ergodic, meaning that the only invariant measures have either full
or zero measure.  This means that as a classical dynamical system, the
geodesic flow mixes things up in phase space.  As eigenfunctions tend
to concentrate along invariant sets, this suggests that the
eigenfunctions must be uniformly distributed in phase space.  Such
results are known as {\it quantum ergodicity}.  First stated by 
\v{S}nirel$'$man\cite{Sni-qe} and proved in the case of negative curvature
by Zelditch \cite{Zel1}, smooth ergodic flow by Colin de Verdi\`ere
\cite{CdV}, and for boundary value problems (such as the Bunimovich
stadium) by G\'erard-Leichtnam \cite{GeLe-qe}; quantum ergodicity
states that a density one sequence of eigenfunctions (or quasimodes) does
equidistribute in phase space.  None of these results precludes the
existence of an exceptional subsequence of eigenfunctions or quasimodes (of density
zero) which scar, or concentrate on smaller invariant sets.

For example, on a hyperbolic cylinder, it is known \cite{CdVP-I} that
quasimodes can concentrate at a logarithmic rate on a single unstable periodic geodesic.
This concentration is known to be sharp as well (see
\cite{Chr-NC,Chr-NC-erratum,Chr-QMNC}).  Other possible concentration
rates are studied for 0-Gevrey smooth partially rectangular billiards
in \cite{Chr-inf-deg} and on Gevrey smooth surfaces of revolution in
\cite{Chr-surf-rev}.
For partially rectangular billiards, there is a relatively large (but
still measure zero) invariant set, referred to as the
``bouncing-ball'' set; the broken periodic geodesics reflecting off of
the flat rectangular part.  It is still measure zero in phase space,
because only the vertical directions remain bouncing-ball
trajectories.  It is a subject of much debate whether there exist
eigenfunctions or quasimodes which concentrate in the rectangular
part, and how fast they concentrate.  A result of Hassell
\cite{Has-nque} informs us that the Bunimovich stadium generically lacks 
quantum {\it unique} ergodicity in the sense that there exist
exceptional sequences of eigenfunctions which do not equidistribute.
This means understanding the methods and location of scarring (or
non-concentration, as in this article) is a very rich subject.

\section{Proof of Theorem \ref{T:stad-non-conc}}
\begin{proof}

The proof proceeds by contradiction.  Suppose the statement is false
and there exists $\delta_0>0$ such that 
\begin{align}
\| u \|_{L^2(\Omega \setminus R)}  & \leq    \lambda^{-2-\delta_0}
\| u \|_{L^2(\Omega) }, \label{E:small-1} \\
\| D_x u \|_{L^2(\Omega \setminus R)} & \leq   \lambda^{-1-\delta_0}
\| u \|_{L^2(\Omega) }, \text{ and }\label{E:small-2} \\
\| D_x^2 u \|_{L^2(\Omega \setminus R)} & \leq   \lambda^{-\delta_0}
\| u \|_{L^2(\Omega) }, \label{E:small-3} 
\end{align}
This means that the quasimode $u$ and its derivatives are small in an
appropriate sense in the wings.  This will lead to a contradiction.

We observe that shrinking $\epsilon_0>0$ or $\delta_0>0$ only strengthens the
statement of the theorem, applying to weaker quasimodes, and
contradicting a weaker statement in (\ref{E:small-1}-\ref{E:small-3}).  Hence we will allow ourselves to shrink
$\epsilon_0>0$ and $\delta_0>0$ several times in the course of the
proof. 

We now straighten the boundary near
the rectangular part so that we may approximately
separate variables as in the proof of \cite[Theorem 3]{Chr-inf-deg}.  As much of
the work is done in that work, we only fill in the details of how to
replace our domain with a 0-Gevrey domain and apply the results of
\cite[Theorem 3]{Chr-inf-deg}.

The
boundary $\Gamma$ near $R$ is given by $y = \pm Y(x) = \pm ( \pi +
r(x) )$ for $x \in [-a-\delta, a + \delta]$ for some $\delta>0$.
Write $P_0 = -\p_x^2 -\p_y^2$ for the flat Laplacian.  

We straighten the boundary near $R$ and compute the corresponding
change in the metric.  From this we will get a non-flat
Laplace-Beltrami operator which is almost separable.  This introduces
some non-trivial curvature, which is unfortunately not smooth, as
$Y(x)$ is not smooth, so we then
conjugate to a new flat problem, and then compare to the 0-Gevrey
case.  
We
change variables $(x, y) \mapsto (x', y')$  locally near the rectangular part:
\[
\begin{cases}
x = x', \\
y =  y' Y(x').
\end{cases}
\]
Thus when $y = \pm Y(x) = \pm Y(x')$, $y' = \pm 1$.  We have
\begin{align*}
g & = dx^2 + dy^2 \\
& = (dx')^2 + (Y dy' + y' Y'(x') dx')^2 \\
& = (1 + A)(dx')^2 + 2B dx' dy' + Y^2 (dy')^2,
\end{align*}
where 
\[
A = (y' Y'(x'))^2,
\]
and
\[
B = y' Y' Y.
\]
We pause to observe that $A$ is quadratic in $y'$ and $Y'(x')$ and $B$
is linear in $y$ and $Y'(x')$.

In matrix notation, 
\[
g = \left( \begin{array}{cc} 1 + A & B \\ B &  Y^2 \end{array}
\right).
\]
Let us drop the cumbersome $(x', y')$ notation and write $(x,y)$
instead.  In order to compute $\Delta_g$ in these coordinates, we need
$| g|$ and $g^{-1}$.  We compute
\begin{align*}
| g | & = Y^2 (1 + A) - B^2 \\
& = Y^2 + y^2 Y^2 (Y')^2 - y^2 Y^2 (Y')^2 \\
& = Y^2.
\end{align*}
Hence
\[
g^{-1} = Y^{-2} \left( \begin{array}{cc} Y^2 & - B \\ -B &  1 + A \end{array}
\right).
\]
For our quasimode $u$ as above, we have after a tedious computation 
\begin{align*}
-\Delta_g u & = -\Big( \p_x^2 + Y^{-2} (1 + A) \p_y^2 + Y' Y^{-1} \p_x
- 2B Y^{-2} \p_x \p_y \\
& \quad 
-Y^{-1} (B/Y)_x \p_y -Y^{-1} (B/Y)_y \p_x + Y^{-1} ((1 + A)/Y)_y \p_y
\Big) u .
\end{align*}
The boundary condition is now 
\[
u|_{ y = \pm 1 } = 0,
\]
locally near the rectangular part $R$.


We observe now that the coefficient of $\p_x$ is
\begin{align*}
Y' Y^{-1} - Y^{-1} (B / Y)_y & = Y^{-1} (Y' - Y') \\
& = 0,
\end{align*}
since $Y$ does not depend on $y$.  The coefficient of $\p_y$ can also
be simplified:
\begin{align*}
Y^{-1} \left( -(B/Y)_x + ((1 + A)/Y)_y \right) & = Y^{-1} \left(
  \frac{-y Y'' Y - y (Y')^2}{Y} + \frac{B Y'}{Y^2} + \frac{ 2y
    (Y')^2}{Y^2} \right) \\
& = Y^{-1} \left( -y Y'' + 2 y (Y')^2 / Y \right) \\
& = -y Y \left( \frac{ Y'}{Y^2} \right)_x.
\end{align*}

The volume element in these coordinates is 
\[
dV = Y dx dy.
\]
We want to conjugate our Laplacian by an isometry of metric spaces to
obtain an (essentially) self-adjoint operator with respect to $dx
dy$.  That is, let $T: L^2(dV) \to L^2(dxdy)$ be given by
\[
Tu(x,y) = Y^{1/2}(x) u(x,y).
\]
Then
\[
-\tDelta = -T \Delta T^{-1}
\]
is essentially self-adjoint on $L^2(dxdy)$.  We compute: 
\begin{align*}
\tDelta & = Y^{1/2} \Delta Y^{-1/2} \\
& =  \p_x^2 + Y^{-2} (1 + A) \p_y^2 -2B Y^{-2} \p_x \p_y - Y' Y^{-1}
\p_x \\
& \quad + \left( B Y' Y^{-3} - y Y (Y'/Y^2)_x \right) \p_y \\
& \quad -\frac{1}{2}
Y'' Y^{-1} + \frac{3}{4} (Y')^2 Y^{-2} .
\end{align*}
The terms $\p_x^2$, $Y^{-2} \p_y^2$, and the potential terms are already in divergence form,
since $Y$ does not depend on $y$.  We now consider the non-divergence
terms to make sure the whole operator is essentially self-adjoint in
simplest terms.  That is, we compute (recalling the forms of $A$ and
$B$) 
\begin{align*}
-& \frac{Y'}{Y} \p_x + \frac{A}{Y^2} \p_y^2 + \left( B Y' Y^{-3} - y Y
  (Y'/Y^2)_x \right) \p_y -2B Y^{-2} \p_x \p_y \\
& = -\frac{Y'}{Y} \p_x + \p_y \frac{A}{Y^2} \p_y - 2y \frac{(Y')^2}{Y^2}
\p_y - y \left( Y^2 (Y'/Y^3)_x \right) \p_y - \p_x \frac{B}{Y^2} \p_y
\\
& \quad +
\left( \frac{B}{Y} \right)_x \p_y - \p_y \frac{B}{Y^2} \p_x+
\left( \frac{B}{Y} \right)_y \p_x \\
& = -\frac{Y'}{Y} \p_x + \p_y \frac{A}{Y^2} \p_y - y \left(  (Y'/Y)_x
\right) \p_y - \p_x \frac{B}{Y^2} \p_y \\
& \quad + y \left(  (Y'/Y)_x
\right) \p_y - \p_y \frac{B}{Y^2} \p_x+
\left( \frac{Y'}{Y} \right) \p_x \\
& = \p_y \frac{A}{Y^2} \p_y - \p_x \frac{B}{Y^2} \p_y - \p_y
\frac{B}{Y^2} \p_x.
\end{align*}
All told then, we have
\[
\tDelta = \p_x^2 + \p_y Y^{-2} (1 + A) \p_y - \p_x \frac{B}{Y^2} \p_y - \p_y
\frac{B}{Y^2} \p_x -\frac{1}{2}
Y'' Y^{-1} + \frac{3}{4} (Y')^2 Y^{-2}.
\]

Let us write $\tOmega$ and $\tR$ for $\Omega$ and $R$ in
these new coordinates.  
We record that the rectangular part now is now $y = \pm 1$, with $-a \leq x
\leq a$, the function $Y(x) = \pi + r(x)$ is $C^{1,1}$ and piecewise $C^\infty$ with $r(x) \equiv
0$ for $x \in [-a,a]$, so that $r(x) = 
O(|\pm x-a|^2)$ as $\pm x \to a+$.  This means the function $A = (y
Y'(x))^2 = y^2 \O( | \pm x - a |^2)$ as $x$ approaches the interval
$[-a,a]$ from without, and the function $B = y Y'(x) Y(x) = y \O( | \pm
x-a|)$.

We return briefly to the $(x, y ) = (x', y' Y(x'))$ notation, where
$(x,y)$ are the coordinates in $\Omega$.  
Writing $v$ for our quasimode in these new coordinates, we observe in our
new coordinates we can write 
\[
u(x,y) = Y^{-1/2}(x) v(x, y Y^{-1}(x) ).
\]
We need to express $u_x$ and $u_{xx}$ in terms of the derivatives of
$v$ so that we may write the conditions
(\ref{E:small-1}-\ref{E:small-3}) in terms of derivatives of $v$.  We
first compute
\begin{align*}
u_x & = -\frac{1}{2} Y' Y^{-3/2} v(x, y Y^{-1} )  + Y^{-1/2} \left(
v_{x'}(x, y Y^{-1} ) + v_{y'}(x, y Y^{-1} ) y \left( -\frac{ Y'}{Y^2}
\right) \right) \\
& =  -\frac{1}{2} Y' Y^{-3/2} v(x, y Y^{-1} )  + Y^{-1/2} \left(
v_{x'}(x, y Y^{-1} ) - v_{y'}(x, y Y^{-1} ) y' \left( \frac{ Y'}{Y}
\right) \right) \\
& =  -\frac{1}{2} Y' Y^{-3/2} v(x', y' )  + Y^{-1/2} \left(
v_{x'}(x', y' ) - v_{y'}(x', y' ) y' \left( \frac{ Y'}{Y}
\right) \right) ,
\end{align*}
by the definitions of $x'$ and $y'$.  
Using this, we next compute
\begin{align*}
u_{xx} & = \left( -\frac{1}{2} Y'' Y^{-3/2} + \frac{3}{4} (Y')^2
  Y^{-5/2} \right) v(x, y Y^{-1} ) \\
& \quad - Y' Y^{-3/2} ( v_{x'}(x, y Y^{-1} ) - v_{y'}(x, y Y^{-1} ) y'
Y' Y^{-1} ) \\
& \quad + Y^{-1/2} \left( v_{x'x'} (x, y Y^{-1} ) - 2 v_{x'y'}(x, y
  Y^{-1} ) y' \frac{Y'}{Y} + v_{y'y'}(x, y Y^{-1} ) \left( y'
    \frac{Y'}{Y} \right)^2 \right) \\
& = \left( -\frac{1}{2} Y'' Y^{-3/2} + \frac{3}{4} (Y')^2
  Y^{-5/2} \right) v(x', y' ) \\
& \quad - Y' Y^{-3/2} ( v_{x'}(x', y'  ) - v_{y'}(x', y' ) y'
Y' Y^{-1} ) \\
& \quad + Y^{-1/2} \left( v_{x'x'} (x', y'  ) - 2 v_{x'y'}(x', y'
   ) y' \frac{Y'}{Y} + v_{y'y'}(x', y' ) \left( y'
    \frac{Y'}{Y} \right)^2 \right),
\end{align*}
again by the definitions of $x'$ and $y'$.  

The assumptions (\ref{E:small-1}-\ref{E:small-3}) now read
\begin{equation}
\label{E:small-1a}
\| v \|_{L^2( \tOmega \setminus \tR)} \leq \lambda^{-2-\delta_0},
\end{equation}
\begin{align}
\Bigg\| &  -\frac{1}{2} Y' Y^{-3/2} v(x', y' )  \notag \\
& \quad + Y^{-1/2} \left(
v_{x'}(x', y' ) - v_{y'}(x', y' ) y' \left( \frac{ Y'}{Y}
\right) \right)  \Bigg\|_{L^2( \tOmega \setminus \tR)} \notag \\
& \leq
\lambda^{-1-\delta_0},
\label{E:small-2a}
\end{align}
and
\begin{align}
\Bigg\| & \left( -\frac{1}{2} Y'' Y^{-3/2} + \frac{3}{4} (Y')^2
  Y^{-5/2} \right) v(x', y' ) \notag \\
& \quad - Y' Y^{-3/2} ( v_{x'}(x', y'  ) - v_{y'}(x', y' ) y'
Y' Y^{-1} ) \notag \\
& \quad + Y^{-1/2} \left( v_{x'x'} (x', y'  ) - 2 v_{x'y'}(x', y'
   ) y' \frac{Y'}{Y} + v_{y'y'}(x', y' ) \left( y'
    \frac{Y'}{Y} \right)^2 \right) \Bigg\|_{L^2( \tOmega \setminus
\tR)} \notag \\
& \leq \lambda^{-\delta_0}.
\label{E:small-3a}
\end{align}

In the sequel, we will be interested in several estimates on various
derivative quantitites in the wings, but localized in a $\lambda$
dependent neighbourhood of $\tR$.  Let $\psi \in \Ci( \reals_x)$ have
support in $\{ | x | \leq a + c \lambda^{-1-\epsilon_0/2} \}$.
Let $R$
be the ramp function
\[
R(x) = \begin{cases}
0, \text{ for } x \leq a, \\
x-a, \text{ for } x \geq 0,
\end{cases}
\]
and let $H = R'$ be the associated Heaviside function.  These
functions will localize in the right wing.  Of course the same
argument can be used to prove estimates in the left wing as well.

\begin{lemma}
\label{L:der-1}
We have the following estimates for our quasimode $v$:
\begin{align}
\| H \psi \p_{y'} v \| & = \O ( \max \{ \lambda^{-1-\delta_0} ,
  \lambda^{-1-\epsilon_0/2 } 
  \} ) 
\label{E:y-der} \\
\| H \psi \p_{x'} v \| & = \O ( \lambda^{-1-\delta_0}
) \label{E:x-der} \\
\| R \psi \p_{x'}\p_{y'} v \| & = \O ( \max \{ \lambda^{-\epsilon_0/2}
, \lambda^{-\delta_0}
\} )
\label{E:xy-der} \\
\| R^2 \psi \p_{y'}^2 v \| & = \O ( \max\{ 
 \lambda^{-2-\delta_0 - \epsilon_0} ,
\lambda^{-2-3\epsilon_0/2}
\} )
\label{E:yy-der} \\
\| H \psi \p_{x'}^2 v \| & = \O ( \max \{ 
\lambda^{-\delta_0}  ,  \lambda^{-\epsilon_0/2} 
\} )
 \label{E:xx-der} .
\end{align}
\end{lemma}

\begin{remark}
We will use repeatedly in the proof that $|R^k \psi | \leq C_k
\lambda^{-k(1 + \epsilon_0/2)} | \psi |$.  In this sense, the
numerology in the first bound in \eqref{E:yy-der} makes intuitive sense ($R^2$
contributes $\lambda^{-2-\epsilon_0}$, 
 $\p_y^2$ contributes
$\lambda^2$, and the estimate on $v$ alone in the wings {\it should}
contribute then $\lambda^{-2-\delta_0}$).

\end{remark}

\begin{proof}
Let us drop the cumbersome $(x',y')$ notation and write $(x,y)$
instead.  Our quasimode $v$ satisfies
\[
\begin{cases}
 -\tDelta v = \lambda^2 v + \tE, \\
v|_{\p \tOmega} = 0,
\end{cases}
\]
where $\| \tE \| = \O( \lambda^{-\epsilon_0 } ) \| v \|$.  Let us also
assume that $\| v \| = 1$ for simplicity.  Our strategy is to use
integrations by parts and the quasimode equation for $v$ to write
(\ref{E:y-der}-\ref{E:yy-der}) in terms of \eqref{E:xx-der}.  Each
estimate will have a power of $\lambda$ plus a term involving
\eqref{E:xx-der}, but with a small coefficient.  This will allow for
us to finally solve for \eqref{E:xx-der}.  Let us pause in passing to
note that, on the support of $\psi$, $\tOmega$ is rectangular so
integration by parts in $y$ alone is allowed.  We will further be able
to integrate by parts in $x$ for terms involving $R \psi$ or $R^2
\psi$, as this function is compactly supported, vanishing at $x = a$
and for $x \geq a + c \lambda^{-1-\epsilon_0/2}$.

We begin with \eqref{E:y-der}.  We write
\begin{align*}
\| H \psi \p_y v \|^2 & = \int \p_y v H \psi^2 \p_y \bar{v} dx dy \\
& = \int (-\p_v^2 v ) H \psi^2 \bar{v} dx dy \\
& = \int (-\tDelta v) H \psi^2 \bar{v} dx dy + \int (\p_x^2 - p_x B
Y^{-2} \p_y - p_y B
Y^{-2} \p_x )v H \psi^2 \bar{v} dxdy \\
& =: I_1 + I_2.
\end{align*}
Computing:
\begin{align*}
I_1 & = \int (\lambda^2 v + \tE ) H \psi^2 \bar{v} dx dy \\
& \leq C ( \lambda^2 \| H \psi v \|^2 + \| \tE \| \| H \psi v \| ) \\
& \leq C ( \lambda^{-2-2\delta_0} + \lambda^{-\epsilon_0 -2-\delta_0}
) \\
& \leq C \max \left\{ \lambda^{-2-2\delta_0},  \lambda^{-\epsilon_0
    -2-\delta_0} \right\}.
\end{align*}
Further,
\begin{align*}
I_2 & = \int (\p_x^2 - \p_x B
Y^{-2} \p_y - \p_y B
Y^{-2} \p_x )v H \psi^2 \bar{v} dxdy  \\
& \leq \| H \psi \p_x^2 v \| \| H \psi v \| \\
& \quad + \left| \int ( \p_x B Y^{-2} + B Y^{-2} \p_x )v H \psi^2
  \p_y \bar{v} dx dy \right| \\
& \quad + C\left| \int \p_x B_y Y^{-2} v H \psi^2 \bar{v} dx dy \right|
\\
& \leq C \Big(   \lambda^{-2-\delta_0} \| H \psi \p_x^2 v \|
+ (\| H \psi v \| + \| R \psi \p_x v \| )   \| H \psi \p_y v \| \\
& \quad +  (\| H \psi v \| + \| R \psi \p_x v \| )   \| H \psi  v \|
\Big) \\
& \leq C  \Big(    \lambda^{-2-\delta_0} \| H \psi \p_x^2 v \|
+ (\lambda^{-2-\delta_0}  + \lambda^{-1-\epsilon_0/2} \| H \psi \p_x v
\| ) (\| H \psi \p_y v \| + \lambda^{-2-\delta_0} )  \Big).
\end{align*}
Here we have used that $| B | \leq C | R |$ and that $| R \psi | \leq
C\lambda^{-1-\epsilon_0/2} | \psi |$.  
Applying Cauchy's inequality (with small parameter on the terms with
$\p_y v$), increasing $C$ as necessary, and
dropping all the terms which are smaller in $\lambda$, we have
\begin{align*}
I_2 & \leq C \lambda^{-2-\delta_0} ( \| H \psi \p_x^2 v \|  +  \| H
\psi \p_y v \| + \lambda^{-2-\delta_0} ) \\
& \quad +  C \lambda^{-1-\epsilon_0/2} \| H \psi \p_x v
\|  (\| H \psi \p_y v \| + \lambda^{-2-\delta_0} ) \\
& \leq C \Big(  \lambda^{-4-2\delta_0}
+ \lambda^{-1-\delta_0} \lambda^{-1}  \| H \psi \p_x^2 v \|
+  \lambda^{-2-\delta_0}  \| H
\psi \p_y v \| \\
& \quad +  \lambda^{-1-\epsilon_0/2} \| H \psi \p_x v
\|  \| H \psi \p_y v \|
+ \lambda^{-3-\delta_0 - \epsilon_0/2} \Big)\\
& \leq C \Big(
 \lambda^{-2-2\delta_0}
+ \lambda^{-2} \| H \psi \p_x^2 v \|^2 \\
& \quad + \lambda^{-2-\epsilon_0} \| H \psi \p_x v\|^2 \Big)
+ \frac{1}{2} \| H
\psi \p_y v \|^2.
\end{align*}

Collecting all terms from $I_1$ and $I_2$ and keeping only the largest
in $\lambda$, we have
\begin{align*}
\| H \psi \p_y v \|^2 & \leq C \Big( \lambda^{-2-2\delta_0} +
\lambda^{-2-\delta_0-\epsilon_0} 
+ \lambda^{-2} \| H \psi \p_x^2 v \|^2 \\
& \quad + \lambda^{-2-\epsilon_0} \| H \psi \p_x v\|^2 \Big)
+ \frac{1}{2} \| H
\psi \p_y v \|^2,
\end{align*}
or, rearranging (with larger $C>0$),
\begin{equation}
\label{E:y-deriv-1}
\| H \psi \p_y v \| \leq C \left(  \lambda^{-1-\delta_0} +
  \lambda^{-1-\epsilon_0/2 - \delta_0/2 } 
+ \lambda^{-1} \| H \psi \p_x^2 v \| 
 + \lambda^{-1-\epsilon_0/2} \| H \psi \p_x v\|   \right) .
\end{equation}

From \eqref{E:small-2a}, we have
\[
\left\| H \psi  \left(-\frac{1}{2} Y' Y^{-3/2} v  + Y^{-1/2} \left(
\p_x v - \p_y v \left( y\frac{ Y'}{Y}
\right) \right)  \right)  \right\| \leq \lambda^{-1-\delta_0},
\]
or, since $Y$ is bounded above and below and $| Y' | \leq C R$ on the
support of $\psi$, we
have using \eqref{E:y-deriv-1},
\begin{align*}
\| H \psi \p_x v \| & \leq C ( \lambda^{-1-\delta_0} + \| R \psi v \| +
\| R \psi \p_y v \| ) \\
& \leq C \left( \lambda^{-1-\delta_0} + \lambda^{-1-\epsilon_0/2} \| H
  \psi \p_y v \| \right) \\
& \leq C \left( \lambda^{-1-\delta_0} + \lambda^{-2-\epsilon_0/2} \| H
  \psi \p_x^2 v \| + \lambda^{-2-\epsilon_0} \| H
  \psi \p_x v \| \right).
\end{align*}
For $\lambda$ sufficiently large, we can solve for the $\p_x$ terms to
get
\begin{equation}
\label{E:x-deriv-1}
\| H \psi \p_x v \| \leq C ( \lambda^{-1-\delta_0} + \lambda^{-2-\epsilon_0/2} \| H
  \psi \p_x^2 v \| ).
\end{equation}
Plugging this back into \eqref{E:y-deriv-1}, we get
\begin{equation}
\label{E:y-deriv-2}
\| H \psi \p_y v \| \leq C \left(  \lambda^{-1-\delta_0} +
  \lambda^{-1-\epsilon_0/2 - \delta_0/2} 
+ \lambda^{-1} \| H \psi \p_x^2 v \| 
  \right) .
\end{equation}
This gives our preliminary estimates for (\ref{E:y-der}-\ref{E:x-der}).

We next use similar integrations by parts arguments to estimate the
second order mixed derivative.  Unfortunately, in this case, the
$x$-derivative can sometimes fall on the $\psi^2$ term, giving both
growth in $\lambda$, and lack of control by $\psi$.  That is, we only
have
\[
| \p_x \psi^2 | \leq C \lambda^{1+ \epsilon_0/2} | \psi |.
\]
Let us choose $\tpsi \in \Ci_c( \reals_x)$ with support in $\{ | x |
\leq a + 2c \lambda^{-1-\epsilon_0/2} \}$ such that $\tpsi \equiv 1$
on $\supp \psi$ and $| \p^k \tpsi | \leq C_k \lambda^{k (1 +
  \epsilon_0/2)}$ as well.  Then we can also say
\[
| \p_x \psi^2 | \leq C \lambda^{1+ \epsilon_0/2} | \psi \tpsi |,
\]
which will sometimes be useful.  However, in order to close our
estimates, we will have to take a loss on terms involving $\tpsi$ and
estimate them instead just using the trivial global quasimode bound,
rather than the improved bound from being localized in the wings.
That is, we will use
\[
\| H \tpsi \p_x^k v \| \leq C \| \p_x^k v \| \leq C \lambda^k
\]
where appropriate.

We compute:
\begin{align*}
\| &  R \psi \p_x \p_y v \|^2  = \int (\p_x \p_y v) R^2 \psi^2 \p_x \p_y
\bar{v} dx dy \\
& = \int \p_x (-\p_y^2v) R^2 \psi^2 \p_x \bar{v} dx dy \\
& = \int \p_x ( -\tDelta v) R^2 \psi^2 \p_x \bar{v} dx dy + \int
\p_x^3 v R^2 \psi^2 \p_x \bar{v} dx dy \\
& \quad - \int ( \p_x^2 B Y^{-2} \p_y  + \p_x \p_y B Y^{-2} \p_x ) v
R^2 \psi^2 \p_x \bar{v} dx dy \\
& =: I_1 + I_2 + I_3.
\end{align*}
We estimate
\begin{align*}
I_1 & = \int \p_x ( \lambda^2 v + \tE) R^2 \psi^2 \p_x \bar{v} dx dy
\\
& = \lambda^2 \| R \psi \p_x v \|^2 - \int \tE \p_x R^2 \psi^2 \p_x
\bar{v} dx dy \\
& \leq C \lambda^{-\epsilon_0} \| H \psi \p_x v \|^2 \\
& \quad - \int \tE ( 2 R' R \psi^2 \p_x + 2 \psi' \psi R^2 \p_x + R^2
\psi^2 \p_x^2 )\bar{v} dx dy \\
& \leq C \lambda^{-\epsilon_0} \| H \psi \p_x v \|^2 \\
& \quad + C \| \tE \| \left( \| R \psi  \p_x v \| + \| R^2 \psi' \psi
  \p_x v \| + \| R^2 \psi \p_x^2 v \| \right) \\
& \leq C \lambda^{-\epsilon_0} \| H \psi \p_x v \|^2 \\
& \quad + C \| \tE \| \left( \lambda^{-1-\epsilon_0/2}\| H \psi  \p_x
  v \|  + \lambda^{-2-\epsilon_0} \| H \psi \p_x^2  v \| \right) ,
\end{align*}
since $| R \psi | \leq \lambda^{-1-\epsilon_0/2} | \psi |$ as usual.
Using the global $\O ( \lambda^{-\epsilon_0})$ bound on $\tE$, we again apply Cauchy's inequality as necessary:
\begin{align*}
I_1 & \leq C \lambda^{-\epsilon_0} \| H \psi \p_x v \|^2 \\
& \quad + C \left( \lambda^{-1-\epsilon_0} ( \lambda^{-\epsilon_0/2}\| H \psi  \p_x
  v \| ) + \lambda^{-1-\epsilon_0} ( \lambda^{-1-\epsilon_0}\| H \psi \p_x^2 
  v \| ) \right) \\
& \leq C \lambda^{-\epsilon_0} \| H \psi \p_x v \|^2 \\
& \quad + C \left( \lambda^{-2-2\epsilon_0} +   \lambda^{-\epsilon_0}\| H \psi  \p_x
  v \|^2 +  \lambda^{-2-2\epsilon_0}\| H \psi
  \p_x^2 v \|^2  \right) \\
& \leq C \left( \lambda^{-2-2\epsilon_0} +   \lambda^{-\epsilon_0}\| H \psi  \p_x
  v \|^2 +  \lambda^{-2-2\epsilon_0}\| H \psi \p_x^2 
  v \|^2  \right).
\end{align*}
Plugging in \eqref{E:x-deriv-1} we get
\begin{align*}
I_1 & \leq C \left(  \lambda^{-2-2\epsilon_0} +
\lambda^{-2-2\delta_0 - \epsilon_0} + \lambda^{-2-2\epsilon_0}\| H \psi \p_x^2 
  v \|^2
      \right).
\end{align*}

For the integral $I_2$, we compute
\begin{align*}
I_2 & = - \int \p_x^2 v ( 2 R' R \psi^2 \p_x + 2 \psi' \psi R^2 \p_x +
R^2 \psi^2 \p_x^2 ) \bar{v} dx dy \\
& \leq C \Big( \| \psi H \p_x^2 v \| \| R \psi \p_x v \|  
+ \| R^2 \psi' \p_x^2 v \| \| H \psi \p_x v \|   
+ \| R \psi \p_x^2 v \|^2 \Big) \\
& \leq C \Big(   \lambda^{-1-\epsilon_0/2} \| \psi H \p_x^2 v \| \| H
\psi \p_x v \|  \\
& \quad + \lambda^{-1-\epsilon_0/2} \| H \tpsi \p_x^2 v \| \| H \psi
\p_x v \| 
+ \lambda^{-2-\epsilon_0} \| H \psi \p_x^2 v \|^2 \Big) \\
& \leq C \Big( \lambda^{1-\epsilon_0/2} \| \psi H \p_x v \| +
\lambda^{-2-\epsilon_0} \| H \psi \p_x^2 v \|^2 \Big) .
\end{align*}
Again using \eqref{E:x-deriv-1} and Cauchy's inequality as necessary,
we have
\begin{align*}
I_2 & \leq C \Big( \lambda^{1-\epsilon_0/2} (  \lambda^{-1-\delta_0} +
\lambda^{-2-\epsilon_0/2} \| H \psi \p_x^2 v \|)   +
\lambda^{-2-\epsilon_0} \| H \psi \p_x^2 v \|^2 \Big) \\
& \leq C \Big( \lambda^{-\delta_0 - \epsilon_0/2}
+\lambda^{-\epsilon_0/2}  (\lambda^{-1-
  \epsilon_0/2   } \| H \psi \p_x^2 v \|)  \\
& \quad +  \lambda^{-2-\epsilon_0} \| H \psi \p_x^2 v\|^2 \Big) \\
& \leq C \Big( \lambda^{-\delta_0 - \epsilon_0/2}   +
\lambda^{-\epsilon_0} + \lambda^{-2-\epsilon_0} \| H
\psi \p_x^2 v\|^2 \Big).
\end{align*}

We now estimate $I_3$:
\begin{align*}
I_3 & = -\int (\p_x^2 B Y^{-2} \p_y + \p_x \p_y B Y^{-2} \p_x )v  R^2
\psi^2 \p_x \bar{v} dx dy \\
& = \int ( \p_x B Y^{-2} \p_y v) (\p_x \psi^2 R^2 \p_x \bar{v} ) dx dy
+ \int (\p_x B Y^{-2} \p_x v ) (\psi^2 R^2 \p_x \p_y \bar{v} ) dx dy
\\
& =: J_1 + J_2.
\end{align*}
Continuing,
\begin{align*}
J_1 & = \int ( (B Y^{-2} )_x \p_y v + B Y^{-2} \p_x \p_y v ) ( 2 R' R \psi^2 \p_x + 2 \psi' \psi R^2 \p_x +
R^2 \psi^2 \p_x^2 ) \bar{v} dx dy \\
& \leq C \Big(  \| R \psi \p_y v \| \| H \psi \p_x v \|
+ \| H \tpsi \p_y v \|   \| R^2 \psi' \psi \p_x v \|
+ \| H \psi \p_y v \| \| R^2 \psi \p_x^2 v \| \\
& \quad + \| R \psi \p_x \p_y v \| \| R \psi \p_x v \|  + \| R \psi
\p_x \p_y v \| \| \tpsi R \p_x v \|  + \| R \psi \p_x \p_y v \| \|
R^2 \psi \p_x^2 v \| \Big).
\end{align*}
Here in the last three terms we have used that $| B | \leq C R$.   
Continuing as before by using (\ref{E:x-deriv-1}-\ref{E:y-deriv-2}) and
Cauchy's inequality with small parameter, we have
\begin{align*}
J_1 & \leq C \Big( \lambda^{-1-\epsilon_0/2} \| H \psi \p_y v \| \| H
\psi \p_x v \|
+ \lambda \lambda^{-1-\epsilon_0/2} \| H \psi \p_x v \|
\\
& \quad + \lambda^{-2-\epsilon_0} \| H \psi \p_y v \| \| H \psi \p_x^2 v \|  + \lambda^{-1-\epsilon_0/2} \| R \psi \p_x \p_y v \| \| H \psi
\p_x v \|    \\
& \quad +  \lambda \lambda^{-1-\epsilon_0/2} \| R \psi
\p_x \p_y v \| 
+ \lambda^{-2-\epsilon_0} \| R \psi \p_x \p_y v \| \|
H \psi \p_x^2 v \| \Big) \\
& \leq C \Bigg( \lambda^{-1-\epsilon_0/2} \left(  \lambda^{-1-\delta_0} +
  \lambda^{-1-\epsilon_0/2 - \delta_0/2} 
+ \lambda^{-1} \| H \psi \p_x^2 v \| 
  \right) \\
& \quad \quad \cdot (\lambda^{-1-\delta_0} +
\lambda^{-2-\epsilon_0/2} \| H \psi \p_x^2 v \| ) \\
& \quad + \lambda^{-\epsilon_0/2}   (\lambda^{-1-\delta_0} +
\lambda^{-2-\epsilon_0/2} \| H \psi \p_x^2 v \| ) \\
& \quad + \lambda^{-2-\epsilon_0}  \left(  \lambda^{-1-\delta_0} +
  \lambda^{-1-\epsilon_0/2 - \delta_0/2} 
+ \lambda^{-1} \| H \psi \p_x^2 v \| 
  \right) \| H \psi \p_x^2 v \| \\
& \quad 
+ \Big( \lambda^{-1-\epsilon_0/2}  \| H \psi
\p_x v \|    +  \lambda^{-\epsilon_0/2} 
+ \lambda^{-2-\epsilon_0}  \|
H \psi \p_x^2 v \| \Big)^2
\Bigg) \\
& \quad 
+ \frac{1}{4} \| R \psi \p_x \p_y v \|^2 \\
& \leq C 
\Bigg(
\lambda^{-1-\delta_0 - \epsilon_0/2 }
+  \lambda^{-2-\epsilon_0} \| H \psi \p_x^2 v \| 
+ \lambda^{-3-\epsilon_0} \| H \psi \p_x^2 v \|^2 \\
& \quad +
\Big( \lambda^{-1-\epsilon_0/2} (\lambda^{-1-\delta_0} +
\lambda^{-2-\epsilon_0/2} \| H \psi \p_x^2 v \| )
+   \lambda^{-\epsilon_0/2} 
+ \lambda^{-2-\epsilon_0}  \|
H \psi \p_x^2 v \| \Big)^2
\Bigg) \\
& \quad 
+ \frac{1}{4} \| R \psi \p_x \p_y v \|^2 \\
& \leq C \Big(  
\lambda^{-\epsilon_0}
+ \lambda^{-3-\epsilon_0} \| H \psi \p_x^2 v \|^2
\Big)
+ \frac{1}{4} \| R \psi \p_x \p_y v \|^2.
\end{align*}

We still have to estimate $J_2$ (here again we use Cauchy's inequality
with small parameter):
\begin{align*}
J_2 & = \int (\p_x B Y^{-2} \p_x v ) (\psi^2 R^2 \p_x \p_y \bar{v} )
dx dy \\
& = \int \left( (BY^{-2})_x \p_x v + B Y^{-2} \p_x^2 v \right) \psi^2
R^2 \p_x \p_y \bar{v} dx dy \\
& \leq C \left( \| R \psi \p_x v \| + \| R^2 \psi \p_x^2 v \| \right)
\| R \psi \p_x \p_y v \| \\
& \leq C  \left( \| R \psi \p_x v \|^2 + \| R^2 \psi \p_x^2 v \|^2
\right) + \frac{1}{4} \| R \psi \p_x \p_y v \|^2 \\
& \leq C \left(    \lambda^{-4-\epsilon_0 - 2\delta_0} +
\lambda^{-4-2\epsilon_0} \| H \psi \p_x^2 v \|^2 \right) 
+ \frac{1}{4} \| R \psi \p_x \p_y v \|^2.
\end{align*}
Collecting the largest terms in $\lambda$ from $J_1$ and $J_2$, we
have
\[
I_3 \leq C ( \lambda^{-\epsilon_0} + \lambda^{-3-\epsilon_0} \| H \psi
\p_x^2 v \|^2 ) + \frac{1}{2} \| R \psi \p_x \p_y v \|^2.
\]
Finally summing $I_1 + I_2 + I_3$ and keeping only the largest terms
in $\lambda$, we get
\begin{align*}
\| R \psi \p_x \p_y v \|^2
& \leq 
C ( \lambda^{-\epsilon_0} + \lambda^{-\delta_0 - \epsilon_0/2} + 
\lambda^{-2-\epsilon_0}  \| H \psi \p_x^2 v \|^2 ) + \frac{1}{2} \| R
\psi \p_x \p_y v \|^2,
\end{align*}
which, after rearranging, implies
\begin{equation}
\label{E:xy-deriv-1}
\| R \psi \p_x \p_y v \| \leq C ( \lambda^{-\epsilon_0/2}  +
\lambda^{-\delta_0/2 - \epsilon_0/4} +
\lambda^{-1-\epsilon_0/2}  \| H \psi \p_x^2 v \| ) .
\end{equation}

Now we can use the triangle inequality, together with the estimates
already proved, to write the \eqref{E:yy-der} in terms of
\eqref{E:xx-der}:
\begin{align}
\| R^2 \psi \p_y^2 v \| & \leq \| R^2 \psi \tDelta v \| + \| R^2 \psi
\p_x^2 v \| + \|R^2 \psi \p_x B Y^{-2} \p_y v \| + \|R^2 \psi \p_y B
Y^{-2} \p_x v \| \notag \\
& \leq  \lambda^2 \| R^2 \psi  v \| + \| R^2 \psi \tE \| + \| R^2 \psi
\p_x^2 v \| \notag  \\
& \quad + \|R^2 \psi \p_x B Y^{-2} \p_y v \| + \|R^2 \psi \p_y B
Y^{-2} \p_x v \|  \notag \\
& \leq C \Big(  
\lambda^{-\epsilon_0} \| H \psi v \|
+ \lambda^{-2-\epsilon_0} \| \tE \|
+ \lambda^{-2-\epsilon_0} \| H \psi \p_x^2 v \|  \notag \\
& \quad 
+ \| R^2 \psi (B Y^{-2})_x \p_y v \|
+ \| R^2 \psi B Y^{-2} \p_x \p_y v \|
+ \| R^2 \psi (B Y^{-2})_y \p_x v \|
\Big)  \notag \\
& \leq C \Big( 
\lambda^{-\epsilon_0} \| H \psi v \|
+ \lambda^{-2-\epsilon_0} \| \tE \|
+ \lambda^{-2-\epsilon_0} \| H \psi \p_x^2 v \|  \notag \\
& \quad 
+ \lambda^{-2-\epsilon_0} \|  H\psi \p_y v \|
+ \lambda^{-2-\epsilon_0} \|  R\psi  \p_x \p_y v \|
+ \lambda^{-3-3\epsilon_0/2} \| H \psi  \p_x v \|
\Big)  \notag \\
& \leq C \Big(
\lambda^{-2-\delta_0 -\epsilon_0} + \lambda^{-2-2\epsilon_0} 
+ \lambda^{-2-\epsilon_0} \| H \psi \p_x^2 v \|  \notag \\
& \quad 
+ \lambda^{-2-\epsilon_0} \left(  \lambda^{-1-\delta_0} +
  \lambda^{-1-\epsilon_0/2 - \delta_0/2} 
+ \lambda^{-1} \| H \psi \p_x^2 v \| 
  \right) \\
& \quad 
+ \lambda^{-2-\epsilon_0}  ( \lambda^{-\epsilon_0/2}  +
\lambda^{-\delta_0/2 - \epsilon_0/4} +
\lambda^{-1-\epsilon_0/2}  \| H \psi \p_x^2 v \| )
\Big)  \notag \\
& \leq C ( \lambda^{-2-\delta_0 - \epsilon_0}  +
\lambda^{-2-3\epsilon_0/2} + \lambda^{-2-\delta_0/2 - 5 \epsilon_0/4 }
+ \lambda^{-2-\epsilon_0} \| H \psi \p_x^2 v \|) \label{E:yy-deriv-1}
\end{align}

We now want to close the loop of our argument by using the a priori
assumed bounds in $u_x$ and $u_{xx}$ together with
(\ref{E:x-deriv-1}-\ref{E:yy-deriv-1}).    That is, from
(\ref{E:small-2a}-\ref{E:small-3a}), using once again that $| Y' |
\leq C R$, we have
\begin{equation}
\label{E:small-2b}
\| H \psi \p_x v \| \leq C (\lambda^{-1-\delta_0} + \| R \psi \p_y v
\| + \| R \psi v \| )
\end{equation}
and
\begin{align}
\| H \psi \p_x^2 v \| & \leq C ( \lambda^{-\delta_0}
+ \| R \psi \p_x \p_y v \|
+ \| R^2 \psi \p_y^2 v \| \notag \\
& \quad 
+ \| R \psi \p_x v \|
+ \| R^2 \psi \p_y v \|
+ \| H \psi v \| ).
\label{E:small-3b}
\end{align}
We first use similar estimates to pull out the appropriate powers of
$R$ in \eqref{E:small-3b} and then plug in (\ref{E:x-deriv-1}-\ref{E:yy-deriv-1}) into \eqref{E:small-3b}:
\begin{align*}
\| H \psi \p_x^2 v \| & \leq C \Big( \lambda^{-\delta_0}
+ \| R \psi \p_x \p_y v \|
+ \| R^2 \psi \p_y^2 v \|  \\
& \quad 
+ \lambda^{-1-\epsilon_0/2}\| H \psi \p_x v \|
+ \lambda^{-2-\epsilon_0} \| H \psi \p_y v \|
+ \lambda^{-2-\delta_0} \Big) \\
& \leq C \Big(
\lambda^{-\delta_0}
+ ( \lambda^{-\epsilon_0/2}  +
\lambda^{-\delta_0/2 - \epsilon_0/4} +
\lambda^{-1-\epsilon_0/2}  \| H \psi \p_x^2 v \| )
\\
& \quad 
+ ( \lambda^{-2-\delta_0 - \epsilon_0}  +
\lambda^{-2-3\epsilon_0/2} + \lambda^{-2-\delta_0/2 - 5 \epsilon_0/4 }
+ \lambda^{-2-\epsilon_0} \| H \psi \p_x^2 v \|)
 \\
& \quad 
+       \lambda^{-1-\epsilon_0/2} ( \lambda^{-1-\delta_0} + \lambda^{-2-\epsilon_0/2} \| H
  \psi \p_x^2 v \| ) \\
& \quad 
+ \lambda^{-2-\epsilon_0} (  \lambda^{-1-\delta_0} + \lambda^{-1-\epsilon_0/2 - \delta_0/2} +
+ \lambda^{-1} \| H \psi \p_x^2 v \| 
  ) \\
& \quad + \lambda^{-2-\delta_0} \Big).
\end{align*}
As usual, keeping only the largest terms in $\lambda$, we have
\[
\| H \psi \p_x^2 v \| \leq C ( \lambda^{-\delta_0}   + \lambda^{-\epsilon_0/2}  +
\lambda^{-\delta_0/2 - \epsilon_0/4} +
+ \lambda^{-1-\epsilon_0/2}  \| H \psi \p_x^2 v \|) .
\]
For $\lambda$ sufficiently large, this implies
\begin{equation}
\label{E:small-3c}
\| H \psi \p_x^2 v \| \leq C ( \lambda^{-\delta_0}   + \lambda^{-\epsilon_0/2}  +
\lambda^{-\delta_0/2 - \epsilon_0/4} +
+ \lambda^{-1-\epsilon_0/2}  \| H \psi \p_x^2 v \|) .
\end{equation}
This is \eqref{E:xx-der}, once we use Cauchy's inequality yet another
time to get
\[
\lambda^{-\delta_0/2 - \epsilon_0/4} \leq \frac{1}{2} (
\lambda^{-\delta_0} + \lambda^{-\epsilon_0/2} ).
\]

The very last step to close the loop is to plug \eqref{E:small-3c}
into (\ref{E:x-deriv-1}-\ref{E:yy-deriv-1}) to recover
\eqref{E:y-der}, \eqref{E:xy-der}, and \eqref{E:yy-der}, and then plug
the necessary estimates into \eqref{E:small-2b} to recover
\eqref{E:x-der}.  This completes the proof.

\end{proof}


We now continue with the proof of Theorem \ref{T:stad-non-conc}.  
Let $\chi \in \Ci_c$ be a smooth function such that $\chi(x) \equiv 1$
on $\{ | x | \leq a 
\}$ with support in
$\{ | x | \leq a + 
\lambda^{-1-\epsilon_0/2} \}.$  In particular,
this means $| \p^m \chi | \leq C_m \lambda^{m(1 + \epsilon_0/2)}$.
Then 
\begin{align*}
-\tDelta \chi v & = - \chi \tDelta v - [\tDelta, \chi] v \\
& = \lambda^2 \chi v  + \O( 
\lambda^{-\epsilon_0} ) \| v \|  - [\tDelta, \chi] v .
\end{align*}
We need to examine the commutator.  We have (in our previous notation) 
\begin{align*}
\| [\p_x^2, \chi] v \| & \leq \| \chi'' v \| + 2 \| \chi' \p_x v \| \\
& \leq C \lambda^{2+\epsilon_0} \|H\psi v \|
+ C \lambda^{1 + \epsilon_0/2}  \| H \psi \p_x v \| \\
& \leq C ( \lambda^{2 + \epsilon_0 -2 - \delta_0} + \lambda^{1 +
  \epsilon_0/2  -1-\delta_0 } ) \| v \| \\
& \leq C \lambda^{-\epsilon_0} \| v \|
\end{align*}
if $\epsilon_0>0$ is sufficiently small that $\epsilon_0/2 - \delta_0
\leq - \epsilon_0$.  Here we have
used \eqref{E:small-1a} and \eqref{E:x-der} in the third line.  
We of course have $[\p_y^2, \chi] = 0$.  Similarly, for the mixed terms in $\tDelta$,
we have
\begin{align*}
\left\| \left[\p_x \frac{B}{Y^2} \p_y , \chi \right] v \right\| &  =
\left\| \frac{B}{Y^2} \chi' \p_y v \right\| \\
& \leq C \| \chi' R \psi \p_y v \| \\
& \leq C \| H \psi 
\p_y v \|\\
& \leq C \max \{ \lambda^{ -1-\delta_0} , \lambda^{ -1 -
  \epsilon_0/2}
\} \| v \| \\
& \leq C \lambda^{-\epsilon_0} \| v \|
\end{align*}
if $\epsilon_0>0$ is sufficiently small (in this case we just need
$\epsilon_0 \leq 1$).  Here we have used
\eqref{E:y-der} in the third line.  
The other mixed term is similarly handled.   
Of course the potential terms also commute with $\chi$.  
This means that $\chi v$ is still an equally good quasimode as $v$.

Now observe that $A = \O( | \pm x-a|^2)$, $B = \O(|\pm x-a|)$ and the
potential terms are bounded, and all have support in $\{ | x | \geq a \}$.  Hence
\begin{align*}
\left\| \p_x \frac{B}{Y^2} \p_y \chi v \right\| & \leq \left\| \left(
    \frac{B}{Y^2} \right)_x \chi \p_y v \right\| + \left\| 
    \frac{B}{Y^2}  \chi' \p_y v \right\| + \left\| 
    \frac{B}{Y^2}  \chi  \p_x \p_y v \right\| \\
& \leq C \left( \| H \psi \p_y v \| + \| R \psi \p_x \p_y v \| \right) \\
& \leq C \Big( 
\max \{ \lambda^{ -1-\delta_0} , \lambda^{ -1 -\epsilon_0/2} \}  
+\max \{ \lambda^{-\epsilon_0/2}, \lambda^{-\delta_0}       \} \Big) \| v \| \\
& \leq C \lambda^{-\epsilon_0/2} \| v \|
\end{align*}
if $\epsilon_0>0$ is sufficiently small (our previous bound of
$\epsilon_0 \leq 2 \delta_0/3$ suffices here).  Here we have used
\eqref{E:y-der} and \eqref{E:xy-der} in the third line.  
Similarly,
\begin{align*}
\| & \p_y A Y^{-2} \p_y \chi v \| \\
& \leq  \| A_y Y^{-2} \p_y \chi v \|
+\| A Y^{-2} \p_y^2 \chi v \| \\
& \leq C \left( \| R^2 \psi \p_y v \| + \| R^2 \psi \p_y^2 v \|
\right) \\
& \leq C \left( \lambda^{-2-\epsilon_0} \max \{ \lambda^{ -1-\delta_0} , \lambda^{ -1 -
  \epsilon_0/2}
\}  
+ \max \{ \lambda^{ -2-\delta_0 - \epsilon_0} , \lambda^{ -2 -3
  \epsilon_0/2}
\}  \right)
\| v \| \\
& \leq C \lambda^{-\epsilon_0} \| v \|
\end{align*}
if the parameters are again chosen small.  Here we have again used
\eqref{E:y-der} as well as \eqref{E:yy-der}.  
The potential terms satisfy
\[
\left\| (-\frac{1}{2}
Y'' Y^{-1} + \frac{3}{4} (Y')^2 Y^{-2}) \chi v \right\| = \O(
\lambda^{-2-\delta_0} ) \| v \|.
\]
Rearranging and plugging these estimates in to $\tDelta$, we get
\begin{equation}
\label{E:chiv-qm}
-(\p_x^2 + Y^{-2}  \p_y^2 -\lambda^2 ) \chi v =
\O( \lambda^{-\epsilon_0/2}) \| \chi v \|,
\end{equation}
since $\| v \| = \| \chi v \| - \O(\lambda^{-2-\delta_0}) \| v \|$.  
We recall for concreteness that we have shrunk $\epsilon_0>0$ as
necessary and the worst estimate comes from 
the terms with $\p_x \p_y$ in $\tDelta$.
  Let us denote $P = -(\p_x^2 + Y^{-2}  \p_y^2 )$.

Now, on the support of $\chi$, the function $Y(x) = \pi + \O( | \pm x -
a|^2) = \pi + \O( \lambda^{-2-\epsilon_0} )$.  Choose a function $\tY$
in the $0$-Gevrey class $\tY \in \mathcal{G}^0_\tau$ for $\tau <
\infty$ (see \cite{Chr-inf-deg}) satisfying $\tY (x) \equiv \pi$ for $x \in [-a,a]$ and
$\tY'(x) < 0$ for $x < -a$, say.  This means that the corresponding
partially rectangular region for $\tY$ opens ``out'' on the left.  Then on the support of $\chi$, we
have $Y^{-2} - \tY^{-2} = \O( \lambda^{-2 - \epsilon_0} )$.  
Quasimodes for the 
operator $\tP = -\p_x^2 - \tY^{-2} \p_y^2$ are studied in detail in
\cite[Theorem 3]{Chr-inf-deg} (recalled in the appendix below), where
it is shown that for a function $\chi v$ with these support properties 
satisfying 
\[
(\tP-\lambda^2) \chi v = \O( \lambda^{-\epsilon_0'}) \|
\chi v \|, 
\]
for any $\epsilon_0'>0$, 
necessarily $\chi v =
\O(\lambda^{-\infty})$.  
Of course in the case at hand, we have
\begin{align*}
(\tP - \lambda^2) \chi v & = (P-\lambda^2) \chi v + ( \tP - P ) \chi v
\\
& = \| (Y^{-2} - \tY^{-2} ) \p_y^2 \chi v \| \\
& \leq C \| R^2 \p_y^2 \psi v \| \\
& = \O( \lambda^{ -2-3\epsilon_0/2} )\|  v \|,
\end{align*}
by our choice of $\epsilon_0>0$.  
This shows  our quasimodes are quasimodes for $\tP$ as well.  
As $\| \chi v \| \geq \| v \| - C
\lambda^{-2-\delta_0} \| v \|$, this is a contradiction.

\end{proof}

\section{Proof of Theorem \ref{C:stad-non-conc-2}}

In this section, we will first prove an analogue of Lemma
\ref{L:der-1} in the case the boundary is $C^{k + \alpha}$ with
$\alpha + k >2$.  The main differences are that the bounds in the
wings will now have exponents {\it smaller} than $2$, and the powers
of the ramp function $R$ will be larger.  The proof has enough subtle
differences that we reproduce it here in this case.


Let $\psi \in \Ci( \reals_x)$ have
support in $\{ | x | \leq a + c \lambda^{-p} \}$, for some $0 < p \leq
1$.  
Let $R$ and $H$ 
be the ramp and Heaviside functions as above.  
For
this Lemma, we assume an analogue of
(\ref{E:small-1a}-\ref{E:small-3a}).  Let $0 \leq s \leq 2$, and
assume 
\begin{equation}
\label{E:small-1a'}
\| v \|_{L^2( \tOmega \setminus \tR)} \leq \lambda^{-s},
\end{equation}
\begin{align}
\Bigg\| &  -\frac{1}{2} Y' Y^{-3/2} v(x', y' )  \notag \\
& \quad + Y^{-1/2} \left(
v_{x'}(x', y' ) - v_{y'}(x', y' ) y' \left( \frac{ Y'}{Y}
\right) \right)  \Bigg\|_{L^2( \tOmega \setminus \tR)} \notag \\
& \leq
\lambda^{1-s},
\label{E:small-2a'}
\end{align}
and
\begin{align}
\Bigg\| & \left( -\frac{1}{2} Y'' Y^{-3/2} + \frac{3}{4} (Y')^2
  Y^{-5/2} \right) v(x', y' ) \notag \\
& \quad - Y' Y^{-3/2} ( v_{x'}(x', y'  ) - v_{y'}(x', y' ) y'
Y' Y^{-1} ) \notag \\
& \quad + Y^{-1/2} \left( v_{x'x'} (x', y'  ) - 2 v_{x'y'}(x', y'
   ) y' \frac{Y'}{Y} + v_{y'y'}(x', y' ) \left( y'
    \frac{Y'}{Y} \right)^2 \right) \Bigg\|_{L^2( \tOmega \setminus
\tR)} \notag \\
& \leq \lambda^{2-s}.
\label{E:small-3a'}
\end{align}

\begin{lemma}
\label{L:der-2}
Assume the boundary $\p \Omega$ is $C^{k, \alpha}$ with $k + \alpha
> 2$ and $k \geq 1$, and set $\gamma = k + \alpha -1$.  
We have the following estimates for our quasimode $v$, assuming the
bounds (\ref{E:small-1a'}-\ref{E:small-3a'}):
\begin{align}
\| H \psi \p_{y'} v \| & = \O ( \max \{ \lambda^{1-s} ,
  \lambda^{-p(2 \gamma -1)}
  \} ) 
\label{E:y-der'} \\
\| H \psi \p_{x'} v \| & = 
\O ( \max \{ \lambda^{1-s} ,
  \lambda^{-p(3 \gamma -1)}
  \} ) 
\label{E:x-der'} \\
\| R^\gamma \psi \p_{x'}\p_{y'} v \| & = \O ( \max \{ \lambda^{2-s-p\gamma}  ,
\lambda^{1-p(2 \gamma-1)},
\lambda^{(2-s-p(3 \gamma-2))/2} \} )
\label{E:xy-der'} \\
\| R^{\gamma +1} \psi \p_{y'}^2 v \| & = \O ( \max\{ 
 \lambda^{2-s-p(\gamma+1)} ,
\lambda^{1-3 p \gamma}
\} )
\label{E:yy-der'} \\
\| H \psi \p_{x'}^2 v \| & = \O ( \max \{ 
\lambda^{2-s}  ,  
\lambda^{1-p(2 \gamma-1)} \} )
 \label{E:xx-der'} .
\end{align}
\end{lemma}

\begin{remark}
We will use repeatedly in the proof that $|R^q \psi | \leq C_q
\lambda^{-qp} | H \psi |$ for any $q \geq 0$.  In this sense, the
numerology in \eqref{E:yy-der'} makes intuitive sense ($R^{\gamma+1}$
contributes $\lambda^{-p(\gamma+1)}$, 
 $\p_y^2$ contributes
$\lambda^2$, and the estimate on $v$ alone in the wings {\it should}
contribute then $\lambda^{-s}$).  

\end{remark}

\begin{proof}
Let us drop the cumbersome $(x',y')$ notation and write $(x,y)$
instead.  Our quasimode $v$ satisfies
\[
\begin{cases}
 -\tDelta v = \lambda^2 v + \tE, \\
v|_{\p \tOmega} = 0,
\end{cases}
\]
where $\| \tE \| = \O( \lambda^{-\epsilon_0 } ) \| v \|$.  Let us also
assume that $\| v \| = 1$ for simplicity.  Our strategy is to use
integrations by parts and the quasimode equation for $v$ to write
(\ref{E:y-der'}-\ref{E:yy-der'}) in terms of \eqref{E:xx-der'}.  Each
estimate will have a power of $\lambda$ plus a term involving
\eqref{E:xx-der'}, but with a small coefficient.  This will allow for
us to finally solve for \eqref{E:xx-der'}.  
We will use that the previously defined functions $A$ and $B$ satisfy
\[
| A | \sim y^2 R^{2\gamma}, \,\,\, |B| \sim y R^{\gamma}.
\]
Unfortunately, a smaller power of $R$ shows up for the $\p_y^2$ terms in our applications
than in the expression for $\p_x^2$ above (that is, $R^{\gamma +1}$ as
opposed to $R^{2 \gamma}$), so we have stated the Lemma for the
smaller power.

We begin with \eqref{E:y-der'}.  We write
\begin{align*}
\| H \psi \p_y v \|^2 & = \int \p_y v H \psi^2 \p_y \bar{v} dx dy \\
& = \int (-\p_v^2 v ) H \psi^2 \bar{v} dx dy \\
& = \int (-\tDelta v) H \psi^2 \bar{v} dx dy + \int (\p_x^2 - p_x B
Y^{-2} \p_y - p_y B
Y^{-2} \p_x )v H \psi^2 \bar{v} dxdy \\
& =: I_1 + I_2.
\end{align*}
Computing:
\begin{align*}
I_1 & = \int (\lambda^2 v + \tE ) H \psi^2 \bar{v} dx dy \\
& \leq C ( \lambda^2 \| H \psi v \|^2 + \| \tE \| \| H \psi v \| ) \\
& \leq C ( \lambda^{2-2s} + \lambda^{-\epsilon_0 -s}
) \\
& \leq C \lambda^{2-2s},
\end{align*}
since $0 \leq s \leq 2$.  
Further, 
\begin{align*}
I_2 
& \leq \| H \psi \p_x^2 v \| \| H \psi v \| \\
& \quad + \left| \int ( \p_x B Y^{-2} + B Y^{-2} \p_x )v H \psi^2
  \p_y \bar{v} dx dy \right| \\
& \quad + C\left| \int \p_x B_y Y^{-2} v H \psi^2 \bar{v} dx dy \right|
\\
& \leq C \Big(   \lambda^{-s} \| H \psi \p_x^2 v \|
+ (\| R^{\gamma-1} H \psi v \| + \| R^\gamma \psi \p_x v \| )   \| H \psi \p_y v \| \\
& \quad +  (\| R^{\gamma -1} H \psi v \| + \| R^\gamma \psi \p_x v \| )   \| H \psi  v \|
\Big) \\
& \leq C  \Big(    \lambda^{-s} \| H \psi \p_x^2 v \|
+ (\lambda^{-s-p(\gamma-1)}  + \lambda^{-p \gamma} \| H \psi \p_x v
\| ) (\| H \psi \p_y v \| + \lambda^{-s} )  \Big).
\end{align*}
Here we have used that $| B | \leq C | R^\gamma |$ and that $| R^\gamma \psi | \leq
C\lambda^{-p\gamma} | \psi |$.  
Applying Cauchy's inequality (with small parameter on the terms with
$\p_y v$), increasing $C$ as necessary, and
dropping all the terms which are smaller in $\lambda$, we have
\begin{align*}
I_2 & \leq C \Big( \lambda^{2-2s}  + \lambda^{-2} \| H \psi \p_x^2 v
\|^2 
+ \lambda^{-2 p \gamma} \| H \psi \p_x v \|^2 \Big)
+ \frac{1}{2} \| H
\psi \p_y v \|^2.
\end{align*}

Collecting all terms from $I_1$ and $I_2$ and keeping only the largest
in $\lambda$, and solving for $\| H \psi \p_y v \|^2$ as before, we have
\begin{equation}
\| H \psi \p_y v \|^2  \leq C \Big( \lambda^{2-2s}  + \lambda^{-2} \| H \psi \p_x^2 v
\|^2 
+ \lambda^{-2 p \gamma} \| H \psi \p_x v \|^2 \Big)
\label{E:y-deriv-1'}
\end{equation}

From \eqref{E:small-2a'}, we have
\[
\left\| H \psi  \left(-\frac{1}{2} Y' Y^{-3/2} v  + Y^{-1/2} \left(
\p_x v - \p_y v \left( y\frac{ Y'}{Y}
\right) \right)  \right)  \right\| \leq \lambda^{1-s},
\]
or, since $Y$ is bounded above and below and $| Y' | \leq C R^\gamma$ on the
support of $\psi$, we
have using \eqref{E:y-deriv-1'},
\begin{align*}
\| H \psi \p_x v \| & \leq C ( \lambda^{1-s} + \| R^\gamma \psi v \| +
\| R^\gamma \psi \p_y v \| ) \\
& \leq C \left( \lambda^{1-s} + \lambda^{-p \gamma} \| H
  \psi \p_y v \| \right) \\
& \leq C \left( \lambda^{1-s} +\lambda^{-2p \gamma} \| H \psi \p_x v \|
+ \lambda^{-1-p \gamma} \| H
  \psi \p_x^2 v \| \right).
\end{align*}
For $\lambda$ sufficiently large, we can solve for the $\p_x$ terms to
get
\begin{equation}
\label{E:x-deriv-1'}
\| H \psi \p_x v \| \leq C ( \lambda^{1-s} + \lambda^{-1-p\gamma} \| H
  \psi \p_x^2 v \| ).
\end{equation}
Plugging this back into \eqref{E:y-deriv-1'}, we get
\begin{equation}
\label{E:y-deriv-2'}
\| H \psi \p_y v \| \leq C \left(  \lambda^{1-s} 
+ \lambda^{-1} \| H \psi \p_x^2 v \| 
  \right) .
\end{equation}
This gives our preliminary estimates for (\ref{E:y-der'}-\ref{E:x-der'}).

We now estimate the mixed second partial as before.  
Again choose $\tpsi \in \Ci_c( \reals_x)$ with support in $\{ | x |
\leq a + 2c \lambda^{-p} \}$ such that $\tpsi \equiv 1$
on $\supp \psi$ and $| \p^k \tpsi | \leq C_k \lambda^{k p}$ as well. 

We compute: 
\begin{align*}
\| &  R^\gamma \psi \p_x \p_y v \|^2  = \int (\p_x \p_y v) R^{2\gamma} \psi^2 \p_x \p_y
\bar{v} dx dy \\
& = \int \p_x (-\p_y^2v) R^{2\gamma} \psi^2 \p_x \bar{v} dx dy \\
& = \int \p_x ( -\tDelta v) R^{2\gamma} \psi^2 \p_x \bar{v} dx dy + \int
\p_x^3 v R^{2\gamma} \psi^2 \p_x \bar{v} dx dy \\
& \quad - \int ( \p_x^2 B Y^{-2} \p_y  + \p_x \p_y B Y^{-2} \p_x ) v
R^{2\gamma} \psi^2 \p_x \bar{v} dx dy \\
& =: I_1 + I_2 + I_3.
\end{align*}
We estimate
\begin{align*}
I_1 & = \int \p_x ( \lambda^2 v + \tE) R^{2\gamma} \psi^2 \p_x \bar{v} dx dy
\\
& = \lambda^2 \| R^\gamma \psi \p_x v \|^2 - \int \tE \p_x R^{2\gamma} \psi^2 \p_x
\bar{v} dx dy \\
& \leq C \lambda^{2-2p\gamma} \| H \psi \p_x v \|^2 \\
& \quad - \int \tE ( 2\gamma R' R^{2\gamma-1} \psi^2 \p_x + 2 \psi' \psi R^{2\gamma} \p_x + R^{2\gamma}
\psi^2 \p_x^2 )\bar{v} dx dy \\
& \leq C \lambda^{2 - 2 p \gamma} \| H \psi \p_x v \|^2 \\
& \quad + C \| \tE \| \left( \| R^{2 \gamma -1} \psi  \p_x v \| + \| R^{2\gamma} \psi' \psi
  \p_x v \| + \| R^{2\gamma} \psi \p_x^2 v \| \right) \\
& \leq C \lambda^{2-2p\gamma} \| H \psi \p_x v \|^2 \\
& \quad + C \lambda^{-\epsilon_0}
 \left( 
\lambda^{-p(2 \gamma-1)}\| H \psi  \p_x
  v \|  + \lambda^{-2p\gamma} \| H \psi \p_x^2  v \| \right) .
\end{align*}
Again applying Cauchy's inequality as necessary and keeping the
largest terms in $\lambda$, we have 
\begin{align*}
I_1 & \leq C \lambda^{2-2p\gamma} ( \lambda^{2-2s} +
\lambda^{-2-2p\gamma} \| H \psi \p_x^2 v \|^2 ) \\
& \quad + C \lambda^{-\epsilon_0} \left(
\lambda^{-p(2 \gamma-1)} ( \lambda^{1-s}  +  \lambda^{-1-p\gamma} \| H
\psi \p_x^2 v \| )
+ \lambda^{-2p\gamma} \| H \psi \p_x^2  v \| \right) \\
& \leq C \Big[  \lambda^{4 -2s - 2 p \gamma} + \lambda^{-4 p \gamma}
\| H \psi \p_x^2 v \|^2 \\
& \quad + \lambda^{1-s-p(2 \gamma-1) - \epsilon_0} 
+ (\lambda^{-1-p\gamma + p - \epsilon_0}) ( \lambda^{-2p\gamma} \| H
\psi \p_x^2 v \| ) \\
& \quad + (\lambda^{-p(2\gamma -1) - \epsilon_0}) ( \lambda^{-2p\gamma} \| H
\psi \p_x^2 v \| )
\Big] \\
& \leq C\Big[
\lambda^{4 -2s - 2 p \gamma} + \lambda^{-4 p \gamma}
\| H \psi \p_x^2 v \|^2 
+ \lambda^{-2p(2
  \gamma -1) - 2 \epsilon_0} \Big].
\end{align*}
Here we have used \eqref{E:x-deriv-1'} and that $0 < p \leq 1$ and $0
\leq s \leq 2$ implies
\[
1-s+p-\epsilon_0 \leq 2-s -\epsilon_0 = 2 - 2s +s - \epsilon_0 \leq 4
-2s - \epsilon_0,
\]
and
\[
-2+2p -2p\gamma - 2 \epsilon_0 \leq -2 p \gamma-2 \epsilon_0 +2s -2s
\leq -2p\gamma -2 \epsilon_0 + 4 -2s.
\]


For the integral $I_2$, we compute
\begin{align*}
I_2 & = - \int \p_x^2 v ( 2\gamma R' R^{2\gamma-1} \psi^2 \p_x + 2 \psi' \psi R^{2\gamma} \p_x +
R^{2\gamma} \psi^2 \p_x^2 ) \bar{v} dx dy \\
& \leq C \Big( \| \psi H \p_x^2 v \| \| R^{2\gamma-1} \psi \p_x v \|  
+ \| R^{2\gamma} \psi' \p_x^2 v \| \| H \psi \p_x v \|   
+ \| R^\gamma \psi \p_x^2 v \|^2 \Big) \\
& \leq C \Big(   \lambda^{-p(2 \gamma-1)} \| \tpsi H \p_x^2 v \| \| H
\psi \p_x v \|  
+ \lambda^{-2 p \gamma} \| H \psi \p_x^2 v \|^2 \Big) \\
& \leq C \Big( \lambda^{2 -p(2 \gamma -1)} \| \psi H \p_x v \| +
\lambda^{-2p\gamma } \| H \psi \p_x^2 v \|^2 \Big) \\
& \leq C  \Big( \lambda^{2 -p(2 \gamma -1)} (\lambda^{1-s} +
\lambda^{-1-p\gamma}   \| H \psi \p_x^2 v \|) +
\lambda^{-2p\gamma } \| H \psi \p_x^2 v \|^2 \Big) \\
 & \leq C  \Big( \lambda^{3-s -p(2 \gamma -1)} +
(\lambda^{1-p(2\gamma-1)})  (\lambda^{-p\gamma}   \| H \psi \p_x^2 v \| )+
\lambda^{-2p\gamma } \| H \psi \p_x^2 v \|^2 \Big),
\end{align*}
where we have again used \eqref{E:x-deriv-1'}.  Applying Cauchy's
inequality as before, we get 
\begin{align*}
I_2 & \leq C 
\big(
\lambda^{3-s -p(2 \gamma -1)}
+ \lambda^{2-2p(2\gamma-1)}
+ \lambda^{-2p\gamma } \| H \psi \p_x^2 v \|^2 \Big)
\end{align*}

We now estimate $I_3$:
\begin{align*}
I_3 & = -\int (\p_x^2 B Y^{-2} \p_y + \p_x \p_y B Y^{-2} \p_x )v  R^{2\gamma}
\psi^2 \p_x \bar{v} dx dy \\
& = \int ( \p_x B Y^{-2} \p_y v) (\p_x \psi^2 R^{2\gamma} \p_x \bar{v} ) dx dy
+ \int (\p_x B Y^{-2} \p_x v ) (\psi^2 R^{2\gamma} \p_x \p_y \bar{v} ) dx dy
\\
& =: J_1 + J_2.
\end{align*}
Recalling the estimates on $B$ in terms of $R$, we have
\begin{align*}
J_1 & = \int ( (B Y^{-2} )_x \p_y v + B Y^{-2} \p_x \p_y v ) \\
& \quad \cdot ( 2\gamma R' R^{2\gamma-1} \psi^2 \p_x + 2 \psi' \psi R^{2\gamma} \p_x +
R^{2\gamma} \psi^2 \p_x^2 ) \bar{v} dx dy \\
& \leq C \Big(  \| R^{3 \gamma-2} \psi \p_y v \| \| H \psi \p_x v \|
+ \| H \tpsi \p_y v \|   \| R^{3 \gamma-1} \psi' \psi \p_x v \| \\
& \quad + \| H \psi \p_y v \| \| R^{3 \gamma -1} \psi \p_x^2 v \| 
+ \| R^\gamma \psi \p_x \p_y v \| \| R^{2\gamma-1} \psi \p_x v \|  \\
& \quad + \| R^\gamma \psi
\p_x \p_y v \| \| \tpsi R^{2 \gamma} \psi' \p_x v \|  + \| R^\gamma \psi \p_x \p_y v \| \|
R^{2\gamma} \psi \p_x^2 v \| \Big) \\
& \leq C \Big[
\lambda \lambda^{-p(3\gamma-2)} \| H \psi \p_x v \|
+ \lambda^{-p(3 \gamma-1)} \| H \psi \p_y v \| \| H \psi \p_x^2 v \|
\\
& \quad + \Big( \lambda^{-p(2 \gamma-1)} \| H \psi \p_x v \|
+ \lambda \lambda^{-p(2\gamma-1)}
+ \lambda^{-2 p\gamma} \| H \psi \p_x^2 v \| \Big) \| R^\gamma \psi
\p_x \p_y v \| \Big] \\
& \leq C \Big[
\lambda^{1-p(3\gamma-2)} ( \lambda^{1-s} + \lambda^{-1-p\gamma} \| H
\psi \p_x^2 v \| ) \\
& \quad 
+ \lambda^{-p(3\gamma-1)} ( \lambda^{1-s} + \lambda^{-1} \| H
\psi \p_x^2 v \| )  \| H \psi \p_x^2 v \| \\
& \quad + \Big( \lambda^{-p(2 \gamma-1)} ( \lambda^{1-s} + \lambda^{-1-p\gamma} \| H
\psi \p_x^2 v \| ) \\
& \quad 
+  \lambda^{1-p(2\gamma-1)}
+ \lambda^{-2 p\gamma} \| H \psi \p_x^2 v \| \Big) \| R^\gamma \psi
\p_x \p_y v \| \Big]
\end{align*}
Here we have used (\ref{E:x-deriv-1'}-\ref{E:y-deriv-2'}) and the global
$\O( \lambda )$ bound on $\| \p_y v \|$.  Keeping only the largest
terms as always and applying 
Cauchy's inequality with small parameter, we have
\begin{align*}
J_1 & \leq C \Big[
\lambda^{2-s-p(3\gamma-2)}
+( \lambda^{-p(3 \gamma-2)}) ( \lambda^{-p\gamma} \| H \psi \p_x^2 v
\| ) \\
& \quad + (\lambda^{-1-p(2\gamma-1)}) ( \lambda^{-p\gamma} \| H \psi \p_x^2 v
\| ) \\
& \quad + \lambda^{-1-p(3 \gamma -1)}\| H \psi \p_x^2 v \|^2 \\
& \quad + (\lambda^{1-p(2\gamma-1)}
+ \lambda^{-2 p\gamma} \| H \psi \p_x^2 v \| ) \| R^\gamma \psi
\p_x \p_y v \| \Big] \\
& \leq C \Big[
\lambda^{2-s-p(3\gamma-2)}
+ \lambda^{-2p(3 \gamma-2)} + \lambda^{2-2p(2\gamma-1)} + \lambda^{-2p\gamma} \| H \psi \p_x^2 v
\|^2  \Big] \\
& \quad 
+ \frac{1}{4}\| R^\gamma \psi
\p_x \p_y v \| .
\end{align*}

We still have to estimate $J_2$ (here again we use Cauchy's inequality
with small parameter and throw out terms lower order in $\lambda$):
\begin{align*}
J_2 & = \int (\p_x B Y^{-2} \p_x v ) (\psi^2 R^{2\gamma} \p_x \p_y \bar{v} )
dx dy \\
& = \int \left( (BY^{-2})_x \p_x v + B Y^{-2} \p_x^2 v \right) \psi^2
R^{2\gamma} \p_x \p_y \bar{v} dx dy \\
& \leq C \left( \| R^{2 \gamma-1} \psi \p_x v \| + \| R^{2\gamma} \psi \p_x^2 v \| \right)
\| R^{\gamma} \psi \p_x \p_y v \| \\
& \leq C  \left( \| R^{2 \gamma -1} \psi \p_x v \|^2 + \| R^{2\gamma}  \psi \p_x^2 v \|^2
\right) + \frac{1}{4} \| R^\gamma \psi \p_x \p_y v \|^2 \\
& \leq C \left(    \lambda^{-2p (2 \gamma -1)}   \| H \psi \p_x v \|^2
  + \lambda^{-4 p \gamma} \| H \psi \p_x^2 v \|^2
\right) 
+ \frac{1}{4} \| R \psi \p_x \p_y v \|^2 \\
& \leq C \left( \lambda^{2-2s-2 p(2 \gamma -1)} + \lambda^{-4 p
    \gamma} \| H \psi \p_x^2 v \|^2 \right) + \frac{1}{4} \| R \psi \p_x \p_y v \|^2.
\end{align*}
Collecting the largest terms in $\lambda$ from $J_1$ and $J_2$, we
have
\begin{align*}
I_3 & \leq C \Big(
\lambda^{2-s-p(3\gamma-2)}
+ \lambda^{-2p(3 \gamma-2)} + \lambda^{2-2p(2\gamma-1)} \\
& \quad + \lambda^{-2p\gamma} \| H \psi \p_x^2 v
\|^2
+\lambda^{2-2s-2 p(2 \gamma -1)} + \lambda^{-4 p
    \gamma} \| H \psi \p_x^2 v \|^2
 \Big) + \frac{1}{2} \| R \psi \p_x \p_y v \|^2 \\
& \leq C \Big(
\lambda^{2-s-p(3\gamma-2)}
+ \lambda^{-2p(3 \gamma-2)} + \lambda^{2-2p(2\gamma-1)} \\
& \quad + \lambda^{-2p\gamma} \| H \psi \p_x^2 v
\|^2  \Big) + \frac{1}{2} \| R \psi \p_x \p_y v \|^2.
\end{align*}
Finally summing $I_1 + I_2 + I_3$ and keeping only the largest terms
in $\lambda$, we get
\begin{align*}
\| R \psi \p_x \p_y v \|^2
& \leq 
C \Big(  
\lambda^{4 -2s - 2 p \gamma} + \lambda^{-4 p \gamma}
\| H \psi \p_x^2 v \|^2 
+ \lambda^{-2p(2
  \gamma -1) - 2 \epsilon_0} \\
& \quad +\lambda^{3-s -p(2 \gamma -1)}
+ \lambda^{2-2p(2\gamma-1)}
+ \lambda^{-2p\gamma } \| H \psi \p_x^2 v \|^2
\\
& \quad +\lambda^{2-s-p(3\gamma-2)}
+ \lambda^{-2p(3 \gamma-2)} + \lambda^{2-2p(2\gamma-1)} \\
& \quad + \lambda^{-2p\gamma} \| H \psi \p_x^2 v
\|^2
\Big) + \frac{1}{2} \| R
\psi \p_x \p_y v \|^2 \\
& \leq C \Big( 
\lambda^{4 -2s - 2 p \gamma} 
+\lambda^{3-s -p(2 \gamma -1)}
+ \lambda^{2-2p(2\gamma-1)} \\
& \quad + \lambda^{-2p\gamma} \| H \psi \p_x^2 v
\|^2\Big) + \frac{1}{2} \| R
\psi \p_x \p_y v \|^2,
\end{align*}
which, after rearranging, implies
\begin{equation}
\label{E:xy-deriv-1'}
\| R \psi \p_x \p_y v \| \leq C \Big( \lambda^{2 -s -  p \gamma} +\lambda^{(3-s -p(2 \gamma -1))/2}
+ \lambda^{1-p(2\gamma-1)} +
\lambda^{-p\gamma}  \| H \psi \p_x^2 v
\| \Big)
) .
\end{equation}

Now we can use the triangle inequality, together with the estimates
already proved, to write the \eqref{E:yy-der'} in terms of
\eqref{E:xx-der'}:
\begin{align}
\| & R^{\gamma +1} \psi \p_y^2 v \| \\
& \leq \| R^{\gamma+1} \psi \tDelta
v \| + \| R^{\gamma +1} \psi
\p_x^2 v \| + \|R^{\gamma +1} \psi \p_x B Y^{-2} \p_y v \| \notag \\
& \quad +
\|R^{\gamma +1} \psi \p_y B
Y^{-2} \p_x v \| \notag \\
& \leq  \lambda^2 \| R^{\gamma +1} \psi  v \| + \| R^{\gamma +1} \psi \tE \| + \| R^{\gamma +1} \psi
\p_x^2 v \| \notag  \\
& \quad + \|R^{\gamma +1} \psi \p_x B Y^{-2} \p_y v \| + \|R^{\gamma +1} \psi \p_y B
Y^{-2} \p_x v \|  \notag \\
& \leq C \Big(  
\lambda^{2 - p(\gamma +1)} \| H \psi v \|
+ \lambda^{-p(\gamma +1)} \| \tE \|
+  \lambda^{-p(\gamma +1)} \| H \psi \p_x^2 v \|  \notag \\
& \quad 
+ \| R^{\gamma +1} \psi (B Y^{-2})_x \p_y v \|
+ \| R^{\gamma +1} \psi B Y^{-2} \p_x \p_y v \| \notag \\
& \quad 
+ \| R^{\gamma +1} \psi (B Y^{-2})_y \p_x v \|
\Big)  \notag \\
& \leq C \Big( 
\lambda^{2 - s-p(\gamma +1)} 
+ \lambda^{-p(\gamma +1) - \epsilon_0 } 
+ \lambda^{-p(\gamma +1)} \| H \psi \p_x^2 v \|  \notag \\
& \quad 
+ \lambda^{-2 p \gamma} \|  H\psi \p_y v \|
+ \lambda^{-p(\gamma+1)} \|  R^\gamma \psi  \p_x \p_y v \|
+ \lambda^{-p(2 \gamma +1)} \| H \psi  \p_x v \|
\Big)  \notag \\
& \leq C \Big(
\lambda^{2 - s-p(\gamma +1)} 
+ \lambda^{-p(\gamma +1) - \epsilon_0 } 
+ \lambda^{-p(\gamma +1)} \| H \psi \p_x^2 v \|  \notag \\
& \quad 
+ \lambda^{-2 p \gamma} ( \lambda^{1-s} + \lambda^{-1} \| H \psi
\p_x^2 v \|) \notag \\
& \quad 
+ \lambda^{-p(\gamma+1)} (  \lambda^{2 -s -  p \gamma} +\lambda^{(3-s
  -p(2 \gamma -1))/2} \\
& \quad 
+ \lambda^{1-p(2\gamma-1)} +
\lambda^{-p\gamma}  \| H \psi \p_x^2 v
\|   )
\Big)  \notag \\
& \leq C \Big( 
\lambda^{2-s-p(\gamma +1)} + \lambda^{1-3 p \gamma} +
\lambda^{-p(\gamma +1)} \| H \psi \p_x^2 v \| 
\Big)
\label{E:yy-deriv-1'}
\end{align}
Here to get the last inequality, we want to absorb the term with
\[
\lambda^{-p(\gamma +1) + (3 -s -p(2 \gamma -1))/2}
\]
in the stated two $\lambda$ terms.  To do this, we yet again appeal to
Cauchy's inequality by observing that the exponent satisfies
\[
-p(\gamma +1) + (3 -s -p(2 \gamma -1))/2 = \frac{1}{2} (
(2-s-p(\gamma+1)) + (1-3 p \gamma)).
\]

We now want to close the loop of our argument by using the a priori
assumed bounds in $u_x$ and $u_{xx}$ together with
(\ref{E:x-deriv-1'}-\ref{E:yy-deriv-1'}).    That is, from
(\ref{E:small-2a'}-\ref{E:small-3a'}), using once again that $| Y' |
\leq C R^\gamma$, we have
\begin{align}
\| H \psi \p_x^2 v \| & \leq C ( \lambda^{2-s}
+ \| R^\gamma \psi \p_x \p_y v \|
+ \| R^{2\gamma} \psi \p_y^2 v \| \notag \\
& \quad 
+ \| R^\gamma \psi \p_x v \|
+ \| R^{2\gamma} \psi \p_y v \|
+ \| R^{\gamma -1}H \psi v \| ).
\label{E:small-3b'}
\end{align}
We first use similar estimates to pull out the appropriate powers of
$R$ in \eqref{E:small-3b'} and then plug in (\ref{E:x-deriv-1'}-\ref{E:yy-deriv-1'}) into \eqref{E:small-3b'}:
\begin{align*}
\| H \psi \p_x^2 v \| & \leq C \Big( \lambda^{2-s}
+ \| R^\gamma \psi \p_x \p_y v \|
+ \lambda^{-p (\gamma -1)}\| R^{\gamma+1} \psi \p_y^2 v \|  \\
& \quad 
+ \lambda^{-p \gamma} \| H \psi \p_x v \|
+ \lambda^{-2p\gamma} \| H \psi \p_y v \|
+ \lambda^{-p(\gamma -1) -s} \Big) \\
& \leq C \Big(
\lambda^{2-s}
+ (
\lambda^{2 -s -  p \gamma} +\lambda^{(3-s -p(2 \gamma -1))/2}
+ \lambda^{1-p(2\gamma-1)} +
\lambda^{-p\gamma}  \| H \psi \p_x^2 v
\| 
 )
\\
& \quad 
+\lambda^{-p (\gamma -1)}
(\lambda^{2-s-p(\gamma +1)} + \lambda^{1-3 p \gamma} +
\lambda^{-p(\gamma +1)} \| H \psi \p_x^2 v \| )
 \\
& \quad 
+       \lambda^{-p \gamma} ( \lambda^{1-s} + \lambda^{-1-p \gamma} \| H
\psi \p_x^2 v \| )
\\
& \quad 
+ \lambda^{-2p\gamma} ( \lambda^{1-s} + \lambda^{-1} \| H
\psi \p_x^2 v \| )
\\
& \quad + \lambda^{ -p(\gamma -1) -s } \Big).
\end{align*}
As usual, keeping only the largest terms in $\lambda$, we have
\[
\| H \psi \p_x^2 v \| \leq C ( 
\lambda^{2-s}
+ \lambda^{(3-s -p(2 \gamma -1))/2}
+ \lambda^{1-p(2\gamma-1)} +
\lambda^{-p\gamma}  \| H \psi \p_x^2 v
\| 
) .
\]
For $\lambda$ sufficiently large, this implies
\begin{equation}
\label{E:small-3c'}
\| H \psi \p_x^2 v \| \leq C  ( \lambda^{2-s}
+ \lambda^{(3-s -p(2 \gamma -1))/2}
+ \lambda^{1-p(2\gamma-1)}) .
\end{equation}
This is \eqref{E:xx-der'}, once we observe that (again by Cauchy's
inequality)
\[
\lambda^{(3-s -p(2 \gamma -1))/2} = \lambda^{\frac{1}{2} ((2-s) +
  (1-p(2 \gamma -1)))} \leq \frac{1}{2} \left(  \lambda^{2-s}
+ \lambda^{1-p(2\gamma-1)}\right)
.  
\]

The very last step to close the loop is to plug \eqref{E:small-3c'}
into (\ref{E:x-deriv-1'}-\ref{E:yy-deriv-1'}) to recover
\eqref{E:y-der'}, \eqref{E:x-der'}, \eqref{E:xy-der'}, and \eqref{E:yy-der'}.  This completes the proof.

\end{proof}


\begin{proof}[Proof of Theorem \ref{C:stad-non-conc-2}]
We follow the proof of Theorem \ref{T:stad-non-conc} and point out
where to make the changes for $C^{k,\alpha}$ boundary with $k +
\alpha>2$.  
The main differences are that, in the notation of the previous
section, the function $Y(x) = \pi + r(x)$ is $C^{k, \alpha}$ and piecewise $C^\infty$ with $r(x) \equiv
0$ for $x \in [-a,a]$, so that $r(x) = 
O(|\pm x-a|^{k + \alpha})$ as $\pm x \to a+$.  This means the function $A = (y
Y'(x))^2 = y^2 \O( | \pm x - a |^{2(k-1 + \alpha)})$ as $x$ approaches the interval
$[-a,a]$ from without, and the function $B = y Y'(x) Y(x) = y \O( | \pm
x-a|^{k-1+\alpha})$.

The proof proceeds by contradiction.  Suppose for some fixed
$\delta_0>0$ (\ref{E:lower-1a}-\ref{E:lower-3a}) are all false with 
\[
s_{\delta_0} = 1+ \max \left\{ \frac{1}{k + \alpha }, \frac{1 +
    \delta_0}{2(k + \alpha) -3} \right\} + \delta_0.
\]
That is, we assume 
\begin{equation}
\label{E:small-3a''}
\| (\lambda^{-1} D_x)^k u \|_{L^2(\Omega \setminus R)}  \leq   C
\lambda^{-s_{\delta_0}} 
\| u \|_{L^2(\Omega) }
\end{equation}
for $k = 0,1,2$.



We change coordinates as in the previous section, and let $v$,
$\tOmega$, $\tR$ be $u$, $\Omega$, and $R$ in the new coordinates as
before.  Our assumptions \eqref{E:small-3a''} on the quasimode $u$
imply that the function $v$ satisfies
(\ref{E:small-1a'}-\ref{E:small-3a'}) with $s = s_{\delta_0}$.
We observe that, since $k + \alpha>2$, shrinking $\delta_0>0$ as
necessary implies $s \leq 2$.  
Let 
\[
p = \max \left\{ \frac{1}{k + \alpha }, \frac{1 +
    \delta_0}{2(k + \alpha) -3} \right\} = \frac{s-\delta_0}{2} < 1.
\]

Let 
$\chi \in \Ci_c$ be a smooth function such that $\chi (x) \equiv
1$ on $\{ | x | \leq a \}$ with support in $\{ | x | \leq a +
\lambda^{-p} \}$.  As before, this implies
\[
| \p^m \chi | \leq C_m \lambda^{mp}.
\]
As in Lemma \ref{L:der-2}, let $\psi \in \Ci_c$ have support in $\{ |
x | \leq a + c \lambda^{-p} \}$ for $c>1$ sufficiently large.  Assume
also that $\psi \equiv 1$ on $\supp \chi$.  


Our strategy is to show that for this $s$, $p$, and $\chi$, that $\chi
v$ is an $\O(\lambda^{-\delta_0}$ quasimode for a similarly modified
operator to that in the proof of Theorem \ref{T:stad-non-conc}.  
As before,
\[
-\tDelta \chi v = \lambda^2 \chi v + \O( \lambda^{-\epsilon_0}) \| v \|
- [\tDelta, \chi] v,
\]
and we need to estimate the terms in the commutator.  Recall our
notation from Lemma \ref{L:der-2} of $\gamma = k + \alpha -1$, the
ramp function $R$ and the Heaviside function $H$.  Let us only examine
the quasimode in the right wing; the analysis in the left wing is
completely analogous.  This merely allows us to substitute $R$ and $H$
where convenient to directly apply Lemma \ref{L:der-2}.  

We have
\begin{align*}
\|[\p_x^2, \chi] v\| & \leq  2\| \chi' \p_x v\|+ \|\chi'' v\| \\
& \leq C ( \lambda^{p} \| H \psi \p_x v \| + \lambda^{2p} \| H \psi v
\| ) \\
& \leq C \max \{ \lambda^{p + 1-s}, \lambda^{p-p(3\gamma-1)},
\lambda^{2p-s} \} \\
& \leq C \max \{ \lambda^{p + 1-s}, \lambda^{p-p(3\gamma-1)} \},
\end{align*}
since $p < 1$.  This expression is bounded by
\[
\O(  \lambda^{-\delta_0} ) \| v \| ,
\]
since 
\[
p + 1-s = -\delta_0
\]
by definition, and
\[
p-p(3\gamma-1) \leq 1-p(2\gamma -1) \leq -\delta_0.
\]

The mixed term commutator is
\begin{align*}
\left\| \left[ \p_x \frac{B}{Y^2} \p_y, \chi \right] v \right\| & =
\left\| \frac{B}{Y^2} \chi' \p_y v \right\| \\
& \leq C \lambda^p \| R^{\gamma} H \psi \p_y v \| \\
& \leq C \lambda^{p - p\gamma} \| H \psi \p_y v \|\\
& \leq C \max \{ \lambda^{1-s-p(\gamma-1)}, \lambda^{-p(3 \gamma -2)}
\}
& = \O ( \lambda^{-\delta_0} ) ,
\end{align*}
since
\[
1-s-p(\gamma-1) \leq 1-s \leq p+1 -s = -\delta_0,
\]
and
\[
-p(3\gamma-2) \leq -\delta_0
\]
was already estimated above.



We still need to control the mixed terms:
\begin{align*}
& \left\| \p_x \frac{B}{Y^2} \p_y \chi v \right\| \\
& \quad    \leq \left\| \left(
    \frac{B}{Y^2} \right)_x \chi \p_y v \right\| + \left\| 
    \frac{B}{Y^2}  \chi' \p_y v \right\| + \left\| 
    \frac{B}{Y^2}  \chi  \p_x \p_y v \right\| \\
& \quad \leq C \left( \| R^{\gamma-1} \chi \p_y v \| + \| R^\gamma
  \chi' \p_y v \| + \| R^\gamma \chi \p_x \p_y v \| \right) \\
& \quad \leq C ( \lambda^{-p(\gamma-1) } \| H \psi \p_y v \| + \|
R^\gamma \psi \p_x \p_y v \| ) \\
& \quad \leq C 
\max\{ \lambda^{-p(\gamma-1) + 1-s}, \lambda^{-p(\gamma-1)
  -p(2\gamma-1)}, \lambda^{2-s-p\gamma}, \lambda^{1-p(2\gamma-1)},
\lambda^{(2-s-p(3\gamma-2))/2} \} \\
& \quad \leq C \lambda^{-\delta_0}.
\end{align*}
To get this last estimate, we use
\[
-p(\gamma-1) + 1-s \leq 1+p-s = -\delta_0,
\]
\[
-p(\gamma-1)
  -p(2\gamma-1) = -p(3\gamma-2) \leq -p(2\gamma-1) \leq -\delta_0,
\]
since $\gamma >1$, 
\[
2-s-p\gamma = 2 -  (p+1+\delta_0)   -p\gamma = 1-\delta_0 -p(1 +
\gamma) \leq -\delta_0,
\]
and
\[
1-p(2\gamma-1) \leq -\delta_0.
\]
For the last exponent, we use Cauchy's inequality again: 
\[
\lambda^{1-\frac{s}{2}-\frac{p}{2}(3\gamma-2)} \leq \frac{1}{2} (\lambda^{2-s-p\gamma}
+ \lambda^{-p(2\gamma-2)} ) \leq C \lambda^{-\delta_0},
\]
since the first term has already been estimated, while
\[
-p(2\gamma -2) = -p(2\gamma-1) + p \leq 1 - p(2\gamma -1) \leq
-\delta_0.
\]

The other mixed term is handled similarly.
We further compute 
\begin{align*}
\left\| \p_y A Y^{-2} \p_y \chi v \right\| 
& \leq C ( \| R^{2\gamma} \psi \p_y v \| + \| R^{2 \gamma} \psi \p_y^2
v \| ) \\
& \leq C ( \lambda^{-2p\gamma} \| H \psi \p_y v \| +
\lambda^{-p(\gamma-1)} \| R^{\gamma+1} \psi \p_y^2 v \| ) \\
& \leq C \max \{ \lambda^{1-s-2p\gamma}, \lambda^{-4p\gamma +p},
\lambda^{2-s-2p\gamma}, \lambda^{1-4p\gamma + p} \} \\
& \leq C \lambda^{-\delta_0},
\end{align*}
since
\[
1-s-2p\gamma \leq 1 + p -s = -\delta_0,
\]
\[
-p(4\gamma -1) \leq -p(2\gamma-1) \leq -\delta_0,
\]
\[
2-s-2p\gamma \leq 2-s-p\gamma \leq -\delta_0,
\]
and
\[
1-p(4\gamma -1) \leq 1-p(2\gamma-1) \leq -\delta_0.
\]


This means that $\chi v$ is an $\O ( \lambda^{- \delta_0 } )$
quasimode for the reduced operator:
\[
(P - \lambda^2) \chi v = -( \p_x^2 + Y^{-2} (x) \p_y^2 - \lambda^2) \chi v = \O (
\lambda^{-\delta_0} ) \| \chi v \|,
\]
since, as before, $\| \chi v \| = \| v \| - \O ( \lambda^{-s} ) \| v
\|$.  Choosing once again a $0$-Gevrey function $\tY$ such that
$\tY(x) \equiv \pi$ for $x \in [-a,a]$ and $\tY'(x) < 0$ for $x <
-a$ in a neighbourhood of the support of $\chi$.  Then, on the support
of $\chi$, we have $|Y^{-2} - \tY^{-2} |\sim R^{\gamma +1}$ (in the
right wing, and again a similar expression holds in the left wing).   We again write $\tP = -(
\p_x^2 + \tY^{-2} (x) \p_y^2 - \lambda^2)$, so that
\begin{align*}
(\tP - \lambda^2 ) \chi v & = ( P - \lambda^2 ) \chi v + (Y^{-2} -
\tY^{-2} ) \p_y^2 \chi v \\
& = \O ( \lambda^{ -\delta_0} ) \| \chi v \| + C \| R^{\gamma +1} \psi
\p_y^2 v \| \\
& = \O ( \lambda^{ -\delta_0 } + \max\{\lambda^{2-s-p(\gamma +1)},
\lambda^{1-3p\gamma} \}) \| \chi v \| \\
& \leq C \lambda^{-\delta} \| \chi v \|,
\end{align*}
by our choice of $p$ and $s$.  
As before, applying \cite[Theorem 3]{Chr-inf-deg}, we have $\chi v =
\O ( \lambda^{-\infty}) \| v \|$, which is a contradiction.

\end{proof}

\appendix

\section{Summary of results from \cite{Chr-inf-deg}}

In this appendix, we very briefly summarize Theorem 3 from
\cite{Chr-inf-deg}, which is used to produce the final contradiction
to prove Theorem \ref{T:stad-non-conc}.  
The main result is that
if a $0$-Gevrey smooth partially rectangular billiard opens ``outward'' in at least one
wing, then any $\O( \lambda^{-\epsilon})$ quasimode must spread to
outside of any $\O(\lambda^{-\epsilon})$ neighbourhood of the
rectangular part.  This result holds for any $\epsilon>0$.

Let $\Omega \subset \reals^2$ be a planar domain as above, but with boundary in the
0-Gevrey class $\GG^0_\tau$ for $\tau < \infty$ (see \cite[Section 2.2]{Chr-inf-deg}.  Let $Y(x) = \pi +
r(x)$ be a graph parametrization of the boundary of $\Omega$ as
above.

\begin{theorem}
\label{T:billiards}
Consider the quasimode problem on $\Omega$:
\[
\begin{cases}
(-\Delta - \lambda^2) u = E(\lambda) \| u \|_{L^2}, \text{ on }
\Omega, \\
B u = 0, \text{ on } \p \Omega,
\end{cases}
\]
where $B = I$ or $B = \p_\nu$ (either Dirichlet or Neumann boundary
conditions).

Assume that $\pm r'(x)>0$ for at least one of $\pm (x \mp a) >0$ (that is, the
boundary curves ``outward'' away from the rectangular part of the
boundary for at least one side).  Fix $\epsilon>0$.  If $E(\lambda) = \O(\lambda^{-\epsilon})$ as $\lambda \to \infty$ and $\WF_{\lambda^{-1}}
u$ vanishes outside a neighbourhood of size
$\O(\lambda^{-\epsilon})$ of $R$, then $u = \O(\lambda^{-\infty})$ on
$\Omega$.

\end{theorem}

\bibliographystyle{alpha}
\bibliography{p-rect-bib}

\end{document}